%% file: main.tex
\documentclass[11pt]{article}
\usepackage[margin =1in]{geometry}
\usepackage{amssymb,amsthm,amsmath}
\usepackage{enumerate}
\usepackage{hyperref}
\usepackage{cite}
\usepackage[capitalize]{cleveref}
\usepackage{enumitem}
\usepackage{color}
\usepackage{cancel}
\usepackage{pgf}
\usepackage{svg}
\usepackage{pgfplots}
\usepackage{import}

\usepackage[utf8]{inputenc} 
\usepackage[T1]{fontenc}    
\usepackage{hyperref}       
\usepackage{url}            
\usepackage{booktabs}       
\usepackage{amsfonts}       
\usepackage{nicefrac}       
\usepackage{microtype}
\usepackage{multirow}
\usepackage{xfrac}
\usepackage{todonotes}
\usepackage{titlesec}
\usepackage{caption}
\usepackage{subcaption}
\usepackage{algorithm}
\usepackage{algorithmic}

\titleformat{\subsubsection}[runin]
{\normalfont\normalsize\bfseries}{\thesubsubsection}{1em}{}

\usepackage{graphicx}

\usepackage{appendix}
\numberwithin{equation}{section}

\newcommand{\ip}[1] {\left\langle #1 \right\rangle }
\newcommand{\norm}[1]{\left \| #1 \right \|}

\newcommand{\inclu}[0] {\ar@{^{(}->}}

\newcommand{\RR}{\mathbb{R}}

\newcommand{\abs}[1]{\left| #1 \right|}

\newcommand{\dom}{\mathrm{dom}\:}
\newcommand{\epi}{\mathrm{epi}\:}

\newcommand{\interior}{\mathrm{int}\:}
\newcommand{\bdry}{\mathrm{bdry}\:}

\newcommand{\argmin}{\operatornamewithlimits{argmin}}

\newcommand{\argmax}{\operatornamewithlimits{argmax}}

\newtheorem{thm}{Theorem}[section]
\newtheorem{theorem}{Theorem}[section]
\newtheorem{definition}[thm]{Definition}
\newtheorem{proposition}[thm]{Proposition}

\newtheorem{lemma}[thm]{Lemma}

\crefname{claim}{claim}{claims}
\Crefname{claim}{Claim}{Claims}
\crefname{lem}{lemma}{lemmas}
\Crefname{lem}{Lemma}{Lemmas}
\crefname{algorithm}{algorithm}{algorithms}
\Crefname{algorithm}{Algorithm}{Algorithms}

\newtheorem{remark}{Remark}

\theoremstyle{remark}

\newcommand{\varspace}{\mathcal{E}}

\newcommand{\defeq}{:=}

\newcommand{\dist}[1]{\mathrm{dist}\left(#1\right)}
\newcommand{\onedist}[2]{\mathrm{d}\left(#1; #2\right)}
\newcommand{\hrzn}[1]{#1^{\infty}}

\newcommand{\bgInner}[1]{S^{\mathrm{in}}_{#1}}

\newcommand{\fBgInner}[1]{F^{\mathrm{in}}_{#1}}
\newcommand{\smOuter}[1]{s^{\mathrm{out}}_{#1}}
\newcommand{\bgOuter}[1]{S^{\mathrm{out}}_{#1}}

\newcommand{\fBgOuter}[1]{F^{\mathrm{out}}_{#1}}

\newcommand{\sIn}[2]{{#1}^{\mathrm{in}}_{#2}}
\newcommand{\sGen}[2]{{#1}^{\mathrm{gen}}_{#2}}
\newcommand{\sOut}[2]{{#1}^{\mathrm{out}}_{#2}}
\newcommand{\price}[1]{\varrho_{#1}}

\newcommand{\core}[1]{C_{#1}}

\newcommand{\cGap}[1]{\mathrm{w}_{#1}}
\newcommand{\cCenter}[1]{x_{#1}}

\newcommand{\scaleOp}[2]{\left[#1\right]_{#2}}
\newcommand{\normE}[1]{\norm{#1}_{\varspace}}

\newcommand{\set}[1]{\left\{ #1 \right\}}

\newcommand{\eps}{\varepsilon}

\usepackage{mathtools}

\DeclareMathOperator{\epiAdd}{\square}
\newcommand{\infConv}[2]{#1 \epiAdd #2}

\usepackage[algo2e, boxruled]{algorithm2e}

\usepackage[T1]{fontenc}
\usepackage{lmodern}
\usepackage[normalem]{ulem}
\begin{document}

	\title{The Optimal Smoothings of Sublinear Functions and Convex Cones}

	 \author{Thabo Samakhoana\footnote{Johns Hopkins University, Department of Applied Mathematics and Statistics, \url{tsamakh1@jhu.edu}} \qquad Benjamin Grimmer\footnote{Johns Hopkins University, Department of Applied Mathematics and Statistics, \url{grimmer@jhu.edu}}}

	\date{}
	\maketitle

	\begin{abstract}
            This paper considers the problem of smoothing convex functions and sets, seeking the nearest smooth convex function or set to a given one. For convex cones and sublinear functions, a full characterization of the set of all optimal smoothings is given. These provide if and only if characterizations of the set of optimal smoothings for any target level of smoothness. Optimal smoothings restricting to either inner or outer approximations also follow from our theory. Finally, we apply our theory to provide insights into smoothing amenable functions given by compositions with sublinear functions and generic convex sets by expressing them as conic sections.
	\end{abstract}

    \input{intro}
    \input{prelim}
    \input{smoothable_sets}

    \input{cones}
    \input{cones_proofs}
    \input{applications}

    \paragraph{Acknowledgements.} This work was supported in part by the Air Force Office of Scientific Research under award number FA9550-23-1-0531. Benjamin Grimmer was additionally supported as a fellow of the Alfred P. Sloan Foundation.

    {\small
    \bibliographystyle{unsrt}
    \bibliography{bibliography}
    }
    \appendix
    \input{appendix}

\end{document}

%% file: intro.tex
\section{Introduction}
The smoothing of nonsmooth functions is a well-established tool in nonsmooth analysis and nonsmooth optimization. Typically, smooth functions exist arbitrarily close to a given nonsmooth but continuous function. Then, a key tradeoff is balancing how smooth of a function one seeks versus how close it ought to be to the original. This smoothing task has found widespread usage:
\begin{itemize}[noitemsep]
    \item Moreau envelopes~\cite{Moreau1965} provide a smoothing by taking an infimal convolution with a quadratic function. These have played key roles in the design and analysis of proximal and subgradient first-order methods~\cite{Beck2009FISTA,Davis2019}. 

    \item Moreau-type smoothings have been considered in several nonconvex settings: Penot and Bougeard~\cite{penot1988approximation} provide infimal smoothings of functions expressed as a difference of two convex functions, for use in optimization. Lasry and Lions~\cite{lasry1986remark} provided smoothings for bounded uniformly continuous functions, for use in studying Hamilton-Jacobi equations. These smoothings were later generalized to a wider class of functions by~\cite{attouch1993approximation}.
    
    \item Particularly friendly smoothings arise for so-called asymptotic functions~\cite{ben-talsmoothing1989}. For example, one can replace minimizing a finite maximum $\max_i G_i(x)$ with the minimizing the smooth function $\eta \log(\sum_{i}\exp(G_i(x)/\eta))$ for some $\eta>0$. Then, accelerated gradient methods can be applied, leading to accelerated guarantees, see \cite{Beck2012}. Such approaches have been competitive in applications like max flows~\cite{Sherman2013}.
    \item Gaussian convolutional smoothing~\cite{NesterovSpokoinyZerothOrder} provides a classic method of inducing smoothness by introducing noise and has found much success in zeroth-order methods, see~\cite{lei2024subdifferentiallypolynomiallyboundedfunctions} and the many references therein.
    \item Conditions closely related to the existence of nearby smooth functions have been used as a novel ``fine-grained'' measure of the complexity of nonsmooth optimization problems~\cite{diakonikolas2024optimizationfinerscalebounded}.
\end{itemize}
Smoothing techniques have also been applied to sets, but not as widely:
\begin{itemize}
    \item Finding the smallest volume ellipsoid enclosing a finite set of points \cite{John1948, Khachiyan1996, NesterovRounding} provides a smoothing of their convex hull. 
    More recently, the set interpolation theory of Luner and Grimmer~\cite{luner2024performanceestimationsmoothstrongly} provides a minimal set (not necessarily an ellipsoid) of fixed smoothness that encloses the same convex hull as the solution to a second-order cone program. 
    \item Smooth approximations of convex bodies in Banach spaces were considered in \cite{deville1998analytic}.
    \item Xu et al.~\cite{xu2025smoothingmovingballsapproximation} recently proposed an algorithm for minimization under conic constraints, indirectly smoothing the underlying cone. Their smoothing is done by smoothing the support function of a compact base of the dual cone and taking level sets of the smooth function.

\end{itemize}

In most of the existing literature, the choice of the smoothing is done in an ad hoc manner. For example, one may choose the Moreau envelope because it is smooth and it preserves minimizers. Note, however, the Moreau envelope does not generally minimize distance to the original function among all equally smooth functions. Similarly, the above-referenced max function smoothing of $z \mapsto \log(\sum_{i=1}^n \exp(z_i))$ is a $1$-smooth and convex over-estimator of the convex function $z \mapsto \max\set{z_1, \dots, z_n}$, but is not the nearest over-estimator with these properties. 

\paragraph{Our Contributions.} The ad hoc nature of existing theory raises the following natural question:
$$
\textbf{How should one optimally smooth convex functions and sets?}
$$
In this paper, \textit{we resolve this question for all sublinear functions and convex cones}, characterizing the whole Pareto frontier of smooth functions and sets optimally trading off smoothness and distance from the given function or set. Our theory at every step puts both function and set smoothing on equal footing, highlighting fundamental symmetries between these two settings. 

As an application of our function smoothing theory, we obtain well-motivated smoothings for a subclass of amenable \cite[Definition 10.23]{rockafellar2009} functions that take the form $g = \sigma \circ G$ where $\sigma$ is a sublinear function and $G$ is a Lipschitz mapping with a Lipschitz continuous Jacobian. This class includes many nonsmooth functions that arise naturally in optimization; for example, minimizing finite maximums $\max\{G_1(x), \dots, G_n(x)\}$, maximum eigenvalue optimization $\lambda_{\max} (G(x))$, and nonlinear regression with an arbitrary norm $\|G(x)\|$. 

As a concrete example, consider the nonsmooth minimization problem
$$\min_{x\in\RR^d}\max\{G_1(x), \dots, G_n(x)\} = \min_{x\in\RR^d}\sigma(G(x))$$
where $\sigma(z) = \max\set{z_1, \dots, z_n}$ and $G(x) = (G_1(x), \dots, G_n(x))$. Beck and Teboulle~\cite{Beck2012} proposed solving this problem by replacing $\sigma\circ G$ with the smooth approximation $f_{\eta} \circ G(x) = \eta\log\left(\sum_i \exp(G_i(x)/\eta)\right)$ and then applying an accelerated gradient method to this smoothed problem. With an appropriate choice of $\eta$, this approach yields an $\eps$-minimizer of $\sigma\circ G$ within $O(\sqrt{(L\varepsilon + \log(n)M^2)}/\varepsilon)$ iterations, whenever each $G_i$ is convex and $G$ is $M$-Lipschitz with $L$-Lipschitz Jacobian. Our theory provides an optimal smoothing $\sGen{f}{\sigma}$ (see Equation~\eqref{eq: maxFuncOptSmoothing}) of $\sigma=\max$. With appropriately chosen $\eta$, minimizing our optimized smoothing $\eta\sGen{f}{\sigma}(G(x)/\eta)$ yields an $\varepsilon$-minimizer of $\sigma\circ G$ within $O(\sqrt{L\varepsilon+ M^2}/\varepsilon)$ iterations. Asymptotically, this is an improvement by a factor of $1/\sqrt{\log(n)}$\footnote{We note that under different assumptions, the log-sum-exp approach may provide stronger guarantees since the function $\log(\sum_i \exp(z_i))$ is also $1$-smooth under the $l_{\infty}$ norm. This nuance is discussed further in Section~\ref{subsec:amenable}.}.

As an application of our cone smoothings, we show in Theorem~\ref{thm: affineInSmoothableConeDescriptionOfConvexSets} that a generic convex set can be decomposed into an intersection of an affine space and a cone that is amenable to our theory. By replacing the cone with its optimal smoothings, we obtain smooth approximations to the original $C$. Such smoothing may be useful for constrained optimization problems where recent algorithms have been developed that benefit explicitly from the smoothness of constraint sets~\cite{liu2024gauges,luner2024performanceestimationsmoothstrongly,samakhoana2024scalableprojectionfreeoptimizationmethods}.

\paragraph{Outline.} Section~\ref{sec:prelim} introduces needed definitions from convex analysis, emphasizing symmetries between quantities defined on functions and sets. Section~\ref{sec:smoothable} presents our core definitions for a (sublinear) function and (conic) set to be smoothable. Then our main results are stated and proven in Section~\ref{sec: cones}, characterizing the set of all optimal smoothings of any sublinear function or convex cone. Finally, Section~\ref{sec: conicSections} applies our main results to smoothing amenable functions and convex sets, illustrating potential benefits from leveraging optimal smoothings over existing ad hoc alternatives.

%% file: prelim.tex
\section{Preliminaries}\label{sec:prelim}
We consider a finite-dimensional Euclidean space $\varspace$ endowed with a norm $\normE{\cdot}$ that is induced by an inner product $\ip{\cdot, \cdot}_{\varspace}$. We will drop the subscript on the inner product as it will always be clear from context. We identify the dual space of $\varspace^*$ with $\varspace$ itself in the usual way. For a linear operator $A: \varspace \to \varspace^\prime$ between two Euclidean spaces, we denote its operator norm by $\norm{A}_{\varspace \to \varspace^\prime}$, i.e.,
\begin{equation*}
    \norm{A}_{\varspace \to \varspace^\prime} = \sup_{\normE{x} \leq 1}\norm{Ax}_{\varspace^\prime} .
\end{equation*}
We denote the convex hull of a given set $C$ by $\mathrm{conv\ } C$.

\paragraph{Distances and Projections.} Let $C \subseteq \varspace$ be nonempty. We denote the interior and boundary of $C$ by $\interior C$ and $\bdry C$ respectively. We denote the Minkowski sum of two sets $C$ and $\Tilde{C}$ by $C + \Tilde{C} \defeq \set{x + \Tilde{x} \mid x \in C, \Tilde{x} \in \Tilde{C}}$ and the Hausdorff distance $\dist{\Tilde{C}, C}$ by
\begin{align*}
    & \dist{x, C} = \inf_{y \in C} \normE{x - y} \ , \\
    & \onedist{\Tilde{C}}{C} \defeq \sup_{\Tilde{x} \in \Tilde{C}} \dist{\Tilde{x}, C} \ , \\
    & \dist{\Tilde{C}, C} \defeq \max\set{\onedist{\Tilde{C}}{C}, \onedist{C}{\Tilde{C}}} \ .
\end{align*}
We denote orthogonal projection onto $C$ by $P_C(x) = \argmin\{\normE{x - y} \mid y\in C\}$. For any $x \in \varspace$ and $r \geq 0$, let $B(x,r) \defeq \set{y \in \varspace \mid \normE{x - y}\leq r}$ denote the closed ball of radius $r$ around $x$.  When $x = 0$, we sometimes write $B_r$ for $B(0, r)$. Thus for any closed $C$, $B_r + C = \set{x + u \mid \normE{u} \leq r, x \in C} = \set{y \in \varspace \mid \dist{y, C} \leq r}$. Note that if $C$ is closed and convex, then 
\begin{equation}\label{eqn: ballPlusSetProjection}
    P_{B(x, r) + C}(y) = x + P_C(y - x) + P_{B_r}(y - x - P_C(y - x))
\end{equation}
for all $x,y \in \varspace$ and $r \geq 0$. This can be easily established from the fact that $\normE{z}(z - P_{B_r}(z)) = (\normE{z} - r)_+z$ for all $z \in \varspace$. Here, and henceforth, $(\alpha)_+ \defeq \max\{0, \alpha\}$ for any $\alpha \in \RR$. 

\paragraph{Cones, Normals, and Horizons.} We say a set $K$ is a convex cone if it is closed under addition and positive rescaling. Given a generic set $C$, for $x \in C$, the associated normal and horizon cones of $C$ at $x$ are
\begin{align*}
    & N_C(x)\defeq \set{\zeta \in \varspace \mid \ip{\zeta, y - x} \leq 0 \text{ for all } y \in C}, \\
    & \hrzn{C}(x) \defeq \set{u \in \varspace \mid x + \alpha u \in C \text{ for all } \alpha \geq 0},
\end{align*}
respectively. $\hrzn{C}$ denotes the set $\bigcup_{x \in C}\hrzn{C}(x)$. We recall that when $C$ is convex and closed, $\hrzn{C}$ is convex and closed and $\hrzn{C} = \hrzn{C}(x)$ for all $x \in C$, see \cite[Theorem 3.6]{rockafellar2009}.

\paragraph{Functions and Associated Sets.}
Let $f: \varspace \to (-\infty, \infty]$ be an extended valued function. We denote the Fenchel conjugate of $f$ by $f^*$. The domain and epigraph are denoted by
\begin{align*}
     \dom f \defeq \set{x \in \varspace \mid f(x) < \infty} \quad \text{and} \quad \epi f \defeq \set{(x, \alpha) \in \varspace \times \RR \mid f(x) \leq \alpha}
\end{align*}
respectively. $f$ is proper if $\dom f \neq \emptyset$. We say $f$ is closed if $\epi f$ is closed\footnote{We always take $\varspace \times \RR$ as a Euclidean space with inner product $\ip{(x, \alpha), (x', \alpha')} = \ip{x, x'} + \alpha \alpha'$.}. The subdifferential of $f$ at $x \in \dom f$ is 
\begin{equation*}
    \partial f(x) \defeq \set{g \in \varspace \mid (g, -1) \in N_{\epi f}(x, f(x))}.
\end{equation*}

The distance between two functions $f : \varspace \to \RR$ and $g: \varspace \to \RR$ is 
\begin{equation}\label{functionDistance}
    \dist{f, g} \defeq \sup_{x \in \varspace}|f(x) - g(x)|.
\end{equation}
For such functions, we write $f \geq g$ to denote that $f(x) \geq g(x)$ for all $x \in \varspace.$

\paragraph{Smooth Functions and Sets.}
For any $\beta>0$, we say a function $f: \varspace \to \RR$ is $\beta$-smooth if $\normE{\nabla f(x) - \nabla f(y)} \leq \beta\normE{x - y}$ for all $x$, $y \in \varspace$, where $\nabla f(x)$ is the gradient of $f$ at $x$.  The function $f(x) = \frac{\beta}{2}\normE{x}^2$ is the archetypal smooth function. In fact, a convex $f$ is $\beta$-smooth if and only if
\begin{equation}\label{eq: smoothnessQuadBound}
    f(y) \leq f(x) + \ip{\nabla f(x), y - x} + \frac{\beta}{2}\normE{y - x}^2 
\end{equation}
for all $x, y \in \varspace$.

For sets, we say a closed convex set $C$ is $\beta$-smooth if $\normE{\zeta - \zeta'} \leq \beta \normE{x - x'}$ for any $x, x' \in C$ and any unit-norm normal vector $\zeta \in N_C(x)$ and $\zeta' \in N_C(x')$. The ball $B(0, 1/\beta)$ for $\beta > 0$ is the archetypal smooth set. The following analog to \eqref{eq: smoothnessQuadBound} holds for a closed convex set $C$,
\begin{equation}\label{setSmoothnessEquivalence}
   C \text{ is $\beta$-smooth} \iff B(x - \zeta/\beta, 1 / \beta) \subseteq C \quad \text{for all unit norm } \zeta \in N_C(x) .
\end{equation}
Relating these two notions of smoothness, \cite[Lemma 7]{liu2024gauges} ensures 
\begin{equation}\label{smoothnessImpliesEpiSmoothness}
    f \text{ is $\beta$-smooth} \implies \epi f \text{ is $\beta$-smooth}
\end{equation}
for all convex $f: \varspace \to \RR$. Note that the converse to \eqref{smoothnessImpliesEpiSmoothness} is not true, see $f_3$ in Figure~\ref{fig:examples-two-norm} for a counterexample. For a more detailed discussion on the smoothness of sets, see \cite{liu2024gauges}.

%% file: smoothable_sets.tex
\section{Smoothable Convex Functions and Sets} \label{sec:smoothable}
Our goal is to find smooth approximations of convex functions and sets. A convex function $f$ is a $\beta$-smoothing of another convex function $g$ if $f$ is $\beta$-smooth and $\dist{f, g} < \infty$. Similarly, a convex set $S$ is a $\beta$-smoothing of another convex set $C$ if $S$ is $\beta$-smooth and $\dist{S, C} < \infty$. Here, our main interest is in understanding when functions and sets admit smoothings and which smoothings minimize the distance $\dist{f, g}$ as $\beta$ varies. The following definition formalizes this. 
\begin{definition}\label{def: funcSmoothabilityDefinition}
    A convex function $g : \varspace \to \RR$ is $(\lambda,\Delta)$-smoothable if for any $\beta > \Delta$, there exists a $\beta$-smooth convex function $f : \varspace \to \RR$ such that $\dist{f, g} \leq \frac{\lambda}{\beta-\Delta}$. Any such $f$ is called a $\beta$-smoothing of $g$ with parameters $(\lambda,\Delta)$ or a $(\lambda, \Delta, \beta)$-smoothing for short.
\end{definition}
Beck and Teboulle \cite[Definition 2.1]{Beck2012} introduced a very similar notion of function smoothability, motivating our definition. The above definition differs from~\cite[Definition 2.1]{Beck2012} in two ways: (i) we require the smooth approximation $f$ to be globally smooth and (2) we only require bounds on $\dist{f, g}$ rather than enforcing separate tolerances bounding the difference above $g-f$ and the difference below $f-g$. We introduce the following analogous definition for sets.
\begin{definition}\label{def: setSmoothabilityDefinition}
    A closed convex set $C$ is $(\lambda,\Delta)$-smoothable if for any $\beta > \Delta$, there exists a $\beta$-smooth convex set $S$ such that $\dist{S, C} \leq \frac{\lambda}{\beta-\Delta}$. Any such $S$ is called a $\beta$-smoothing of $C$ with parameters $(\lambda,\Delta)$, or a $(\lambda, \Delta, \beta)$-smoothing for short.
\end{definition}

Note that not all functions are smoothable. Consider, for instance, $g(x)=x^4$. Since every $\beta$-smooth function $f$ is upper bounded by $f(0) + f'(0)x + \frac{\beta}{2}x^2$, the difference $g-f$ must diverge as $x$ grows, making $\dist{f, g}=\infty$. Similar reasoning also establishes the necessity of restricting to $\beta>\Delta$ in the above definition. Consider, for instance $g(x) = \frac{\Delta}{2}x^2+|x|$ for any $\Delta>0$. The above reasoning establishes a $\beta$-smoothing $f$ with $\dist{f, g}<\infty$ can only exist for $\beta> \Delta$. In contrast, every convex set is smoothable. Indeed, for any closed, convex $C$ and $\beta > 0$, the set $S \defeq B_{1/\beta} + C$ is $\beta$-smooth and satisfies $\dist{C, S} \leq 1/\beta$. The fact that $B_{1/\beta} + C$ is smooth is a consequence of a more general result: The (Minkowski) sum of two closed convex sets is smooth whenever one of them is smooth, see \cite[Lemma 9]{liu2024gauges}.

\subsection{Smoothings of Sublinear Functions and Convex Cones}
By restricting to sublinear functions $\sigma$ and convex cones $K$, the above notions of smoothability can be simplified, setting $\Delta=0$ without loss of generality. To see this, consider the following simple rescaling operations: for any function $f: \varspace \to \RR$, set $C$, and $\eta>0$, define
\begin{equation}\label{eq: scaleOpDefinition}
    \scaleOp{f}{\eta}(x) = \eta f(x/\eta) \qquad \scaleOp{C}{\eta} = \eta C = \set{\eta x \mid x \in C} .
\end{equation}
Note that $\epi \scaleOp{f}{\eta} = \eta \cdot \epi f = \scaleOp{\epi f}{\eta}$. Therefore, the functional and set rescalings above are one and the same when viewing functions equivalently through their epigraphs. 

Leveraging positive homogeneity, the following lemma characterizes the smoothness of $\scaleOp{f}{\eta}$ and $\scaleOp{C}{\eta}$ and how their distance to a given sublinear $\sigma$ and cone $K$ scales with $\eta$.
\begin{lemma}\label{lem: scalingOperation}
    For any $\eta>0$, function $f$ and sublinear $\sigma$, and set $C$ and cone $K$,
    \begin{equation}\label{scaledDistance}
       \dist{\scaleOp{f}{\eta}, \sigma} = \eta \cdot \dist{f, \sigma} \quad \text{and} \quad  \dist{\scaleOp{C}{\eta}, K} = \eta \cdot \dist{C, K} .
    \end{equation}
    In addition, if $f$ and $C$ are $\beta$-smooth, then $\scaleOp{f}{\eta}$ and $\scaleOp{C}{\eta}$ are $\beta/\eta$-smooth respectively. 
\end{lemma}
\begin{proof}
    We have
    \begin{align*}
        \dist{\scaleOp{f}{\eta}, \sigma} = \sup_{x \in \varspace}\abs{\eta f(x/\eta) - \sigma(x)} = \sup_{x \in \varspace}\eta\abs{f(x/\eta) - \sigma(x/\eta)} = \eta\dist{f, \sigma}
    \end{align*}
    where the third equality uses that $\lambda \sigma(x) = \sigma(\lambda x)$ whenever $\lambda \geq 0$. Similarly
    \begin{align*}
        \dist{\scaleOp{C}{\eta}, K} = \dist{\eta C, K} = \dist{\eta C, \eta K} = \eta\dist{C, K}
    \end{align*}
    where the third equality follows because $\lambda K = K$ whenever $\lambda > 0$. For smoothness, we have
    \begin{align*}
       \normE{\nabla \scaleOp{f}{\eta}(x) - \nabla \scaleOp{f}{\eta}(y)} = \normE{\nabla f(x/\eta) - \nabla f(y/\eta)} \leq (\beta/\eta) \normE{x - y}
    \end{align*}
    where the equality follows by the chain rule and the inequality by $\beta$-smoothness of $f.$ Similarly, for any unit norm $\zeta_1 \in N_{\scaleOp{C}{\eta}}(x)$ and $\zeta_2 \in N_{\scaleOp{C}{\eta}}(y)$ we have
    \begin{align*}
        \normE{\zeta_1 - \zeta_2} \leq \beta \normE{x/\eta - y/\eta} = (\beta/\eta)\normE{x - y}
    \end{align*}
    where the inequality follows from $\beta$-smoothness of $C$ and the identity $N_{\scaleOp{C}{\eta}}(z) = N_C(z / \eta)$.
\end{proof}

From Lemma~\ref{lem: scalingOperation}, it is clear that if a sublinear function $\sigma$ or a cone $K$ is $(\lambda,\Delta)$-smoothable, then it must be $(\lambda,0)$-smoothable. Hence, going forward in our study of optimal smoothings of sublinear functions and convex cones, it suffices to fix $\Delta=0$. In this case, we will abuse notation and say $\sigma$ or $K$ is $\lambda$-smoothable to denote $(\lambda,0)$-smoothability.

\subsection{Inner and Outer Smoothings}
In Definitions~\ref{def: funcSmoothabilityDefinition}~and~\ref{def: setSmoothabilityDefinition}, there is no distinction on whether the approximation is from above or below, or inside or outside. It can sometimes be useful in practice to make such a distinction. This leads to the following definitions for function and set smoothing from inside and from outside. When applying notions of ``inner''/``outer'' approximations to functions, we take the convention of identifying each function with its epigraph. Hence, an inner approximation $f$ of $g$ is a function with $\epi f\subseteq \epi g$. That is, $f\geq g$. 
\begin{definition}\label{def: funcInnerSmoothabilityDefinition}
    A convex function $g$ is inner (or outer) $(\lambda,\Delta)$-smoothable if for any $\beta > \Delta$, there exists a $(\lambda,\Delta,\beta)$-smoothing $f$ such that $f \geq g$ (or $f\leq g$). 
\end{definition}
\begin{definition}\label{def: setInnerSmoothabilityDefinition}
    A convex set $C$ is inner (or outer) $(\lambda,\Delta)$-smoothable if for any $\beta > \Delta$, there exists a $(\lambda,\Delta,\beta)$-smoothing $S$ such that $S \subseteq C$ (or $S \supseteq C$). 
\end{definition}
As noted in the previous subsection, for sublinear functions and convex cones, $\Delta$ can be set to zero in all the above definitions without loss of generality. In such settings, we abuse notation and omit $\Delta=0$ from our statements.

\subsection{Example: Smoothing the Euclidean Norm and the Second Order Cone}
To illustrate these definitions, we briefly consider smoothing the standard Euclidean norm $\sigma(x) = \norm{x}_2$ and the second-order cone $K \defeq \epi \norm{\cdot}_2$. First, we focus on smoothing $\sigma$. Consider the following five candidate smoothings of the Euclidean norm
\begin{align*}
    &f_1(x) := \begin{cases}
        \frac{1}{2}\norm{x}_2^2 & \text{if } \norm{x}_2 \leq 1 \\
        \norm{x}_2- \frac{1}{2} & \text{otherwise}
    \end{cases} & f_2(x) := f_1(x) + \frac{1}{4} \\
    &f_3(x) := \begin{cases}
        2\sqrt{2} - \sqrt{8 - \norm{x}_2^2} & \text{if } \norm{x}_2\leq 2 \\
        \norm{x}_2 - 2(2- \sqrt{2})  & \text{otherwise}
    \end{cases} & f_4(x) := f_3(x) + 6\sqrt{2} - 8  \\
    &f_5(x) := \sqrt{1 + \norm{x}_2^2} - 1
\end{align*}
where $f_1$ and $f_2$ are the Moreau envelope of $\sigma$ (known as a Huber function) and a translation of it, $f_3$ and $f_4$ are functions whose epigraph is a translation of $\epi\sigma + B_{2\sqrt{2}}$, and $f_5$ is an ad hoc choice avoiding piecewise definitions.
One can verify that $f_1,f_3,f_5$ are all outer smoothings while $f_2$ and $f_4$ are neither inner nor outer smoothings.

\begin{figure}[t]
  \begin{minipage}[t]{0.42\textwidth}
    \vspace{0pt}  
    \centering
    \includegraphics[width=\linewidth]{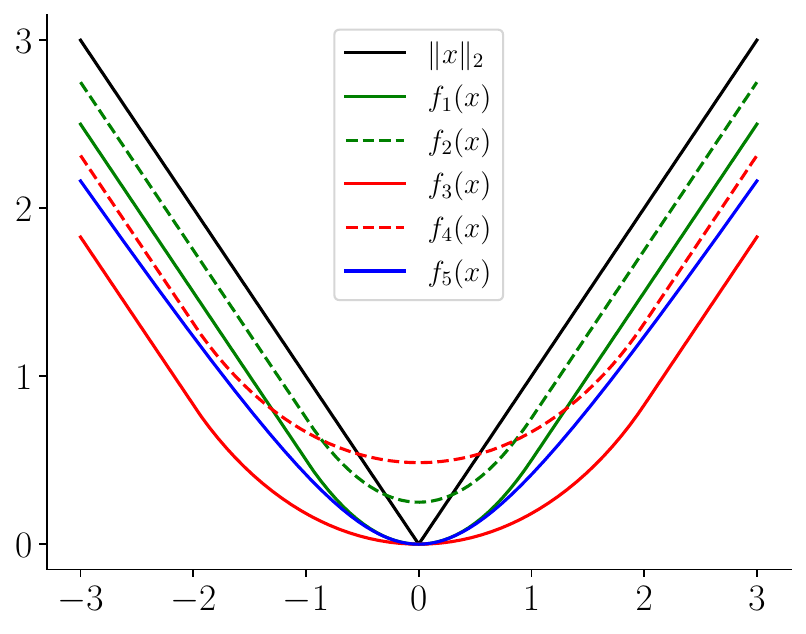}
  \end{minipage}%
  \hfill
  \begin{minipage}[t]{0.56\textwidth}
    \vspace{0pt}  
    \centering
    \begin{tabular}{cccc}
      \toprule
      $f_i$   & $\beta_i$ & $D_i$   & $\lambda_i$ \\
      \midrule
      $f_1$   & $1$               & $1/2$          & $1/2$          \\
      $f_2$   & $1$               & $1/4$          & $1/4$          \\
      $f_3$   & $1$               & $2(2-\sqrt{2})$     & $2(2-\sqrt{2})$     \\
      $f_4$   & $1$               & $\sqrt{2}(6\sqrt{2} - 8)$        & $\sqrt{2}(6\sqrt{2} - 8)$        \\
      $f_5$   & $1$               & $1$                 & $1$                 \\
      \midrule
      $S_i$   & $\beta_i'$ & $D_i'$   & $\lambda_i'$ \\
      \midrule
      $S_1$   & $1$               & $1/2\sqrt{2}$       & $1/2\sqrt{2}$       \\
      $S_2$   & $1$               & 1/4                  &       1/4            \\
      $S_3$   & $1/2\sqrt{2}$     & $2(\sqrt{2}-1)$       & $(\sqrt{2}-1)/\sqrt{2}$ \\
      $S_4$   & $1/2\sqrt{2}$     & $6\sqrt{2} - 8$                 & $(\sqrt{2}-1)/(1+\sqrt{2})$                   \\
      $S_5$   & $1$               & $1/\sqrt{2}$        & $1/\sqrt{2}$        \\
      \bottomrule
    \end{tabular}
  \end{minipage}%
  \caption{%
    Left: Five candidate smoothings of $\norm{x}_2$ in 1D. Right: Smoothness and distance bounds, and the implied smoothability bound for the two-norm and second-order cone from each candidate.
  }
  \label{fig:examples-two-norm}
\end{figure}

For each $i=1,2,\dots, 5$, Figure~\ref{fig:examples-two-norm} computes constants for the $\beta_i$-smoothness of each $f_i$ and the distance $D_i=\dist{f_i, \sigma}$. Since $\sigma=\norm{\cdot}_2$ is sublinear, Lemma~\ref{lem: scalingOperation} ensures $\scaleOp{f_i}{\eta}$ is $\beta_i/\eta$ smoothing with distance from $\sigma$ of $\eta D_i$. As a result, each $f_i$ proves $\sigma$ is $\lambda_i=\beta_i D_i$-smoothable. Among these choices, $f_2$ provides the best smoothability constant in general and $f_1$ provides the best constant among the outer smoothings. Our Theorems~\ref{thm: optimalSigmaSmoothing} and~\ref{thm: optimalSigmaOuterSmoothing} establish that, in fact, these are the unique optimal smoothing and outer smoothing of the Euclidean norm among all $1$-smooth functions.

Each of these function approximations provides a smooth set approximating the second-order cone $K=\epi \sigma$ given by $S_i = \epi f_i$. Each of these sets has some associated smoothness constant $\beta_i'$ and distance $D_i' = \dist{S_i,K}$, also computed in Figure~\ref{fig:examples-two-norm}. Again by Lemma~\ref{lem: scalingOperation}, we know $\scaleOp{S_i}{\eta}$ is $\beta_i'/\eta$ smoothing with distance from $K$ of $\eta D_i'$. Hence, each $S_i$ proves the second-order cone $K$ is smoothable with constant $\lambda_i' = \beta_i' D_i'$. Among these, $S_4$ provides the best smoothability constant in general and $S_3$ provides the best constant among the outer smoothings. Our Theorems~\ref{thm: optimalConeSmoothing} and~\ref{thm: optimalConeOuterSmoothing} establish that these are the unique optimal smoothings.

After developing our theory for generic sublinear functions and convex cones, more interesting examples are considered in Section~\ref{subsec: examples}. Unlike our simple example here, optimal smoothings in general are not unique; in such settings, we provide a characterization of all optimal smoothings. The fact that the Moreau envelope of the two-norm plus a constant was an optimal function smoothing is shared among norms in general, but does not always provide optimal smoothings for sublinear functions.

%% file: cones.tex
\section{Optimal Smoothings of Sublinear Functions and Cones}\label{sec: cones}
As our main result, we characterize the set of all optimal smoothings of any sublinear function or convex cone. Recall that $\sigma : \varspace \to \RR$ is sublinear if 
\begin{align*}
    \sigma(\lambda x) = \lambda \sigma(x)\mathrm{\ if\ } \lambda \geq 0, \qquad \qquad \sigma (x+y) \leq \sigma(x) + \sigma(y)
\end{align*}
and $K \subseteq \varspace$ is a convex cone if 
\begin{align*}
    \lambda x \in K \mathrm{\ if\ } x \in K, \: \lambda \geq 0, \qquad \qquad x+y\in K \mathrm{\ if\ } x,y\in K .
\end{align*}
For the rest of the paper, $\sigma$ will denote a sublinear function and $K$ a closed, convex cone, with a nonempty interior such that $K \neq \varspace$. We define optimal smoothings of functions and sets as follows.
\begin{definition}
    A convex function $f: \varspace \to \RR$ is an optimal $\beta$-smoothing of $g:\varspace \to \RR$ if $f$ is $\beta$-smooth and $\dist{g, f} \leq \dist{g, \Tilde{f}}$ for all $\beta$-smooth, convex $\Tilde{f}: \varspace \to \RR$.
\end{definition}
\begin{definition}
    A closed convex set $S$ is an optimal $\beta$-smoothing of another set $C$ if $S$ is $\beta$-smooth and $\dist{C, S} \leq \dist{C, \Tilde{S}}$ for all $\beta$-smooth, convex $\Tilde{S}$.
\end{definition}

Recall that as an immediate consequence of Lemma~\ref{lem: scalingOperation}, it suffices to characterize all optimal smoothings of $\sigma$ and $K$ at the fixed smoothness level $\beta = 1$. To construct such smoothings, we use the following key objects, which we call the (``functional'' for $\sigma$ or ``conic'' for $K$)  core, center, and width, respectively:
\begin{equation*}
    \begin{matrix}
        \core{\sigma} \defeq \set{(x, r) \in \varspace \times \RR \mid (x, r) + \epi\frac{1}{2}\normE{\cdot}^2 \subseteq \epi \sigma},  & & \core{K} \defeq \set{x \in \varspace \mid x + B(0, 1) \subseteq K}, \\
        & & \\
        (\cCenter{\sigma}, r_\sigma) \defeq \underset{(x, r) \in C_\sigma}{\argmin \: } r + \frac{1}{2}\normE{x}^2, & & x_K \defeq \underset{x \in \core{K}}{\argmin \: }  \normE{x}, \\
        & & \\
        \cGap{\sigma}  \defeq r_\sigma + \frac{1}{2}\normE{x_{\sigma}}^2, & & w_K \defeq \normE{x_K} - 1.
    \end{matrix}
\end{equation*}

We claim the functional and conic cores are both closed and convex. The following pair of lemmas establishes this by giving equivalent dual formulas for both objects as a maximum of affine functions and as an intersection of halfspaces. In particular, the functional core $\core{\sigma}$ is the epigraph of the following closed, convex function:
\begin{equation}\label{priceFormula}
     \price{\sigma}(x) \defeq \underset{\zeta \in \partial \sigma(0)}{\max} \ip{\zeta, x} + \frac{1}{2}\normE{\zeta}^2 \ .
\end{equation}
\begin{lemma}\label{lem: coreSigmaIsPriceEpigraph}
    For any sublinear function $\sigma$, $\core{\sigma} = \epi \price{\sigma}$.
\end{lemma}
\begin{proof}
Observe that it suffices to show that $\price{\sigma}$ equals $\sup_{y \in \varspace} \sigma(y) - \frac{1}{2}\normE{x - y}^2$. This is sufficient since for any $(x, r) \in \varspace \times \RR$,
\begin{align*}
    (x, r) \in \core{\sigma} 
     \iff r + \frac{1}{2}\normE{x - y}^2 \geq \sigma(y) \text{ for all } y \in \varspace   
    & \iff r \geq \sup_{y \in \varspace} \sigma(y) - \frac{1}{2}\normE{x - y}^2 \\
    & \iff (x, r) \in \epi \price{\sigma}
\end{align*}
where the first equivalence follows by definition of $\core{\sigma}$ while the last one uses the claimed equality. Now
\begin{align*}
    \sup_{y \in \varspace} \sigma(y) - \frac{1}{2}\normE{x - y}^2 
    &= \sup_{y \in \varspace} \sup_{\zeta \in \partial \sigma(0)} \ip{\zeta, y} - \frac{1}{2}\normE{x - y}^2 \\
    &= \sup_{\zeta \in \partial \sigma(0)} \ip{\zeta, x} + \sup_{y \in \varspace} \ip{\zeta, y - x} - \frac{1}{2}\normE{y - x}^2 \\
    & = \sup_{\zeta \in \partial \sigma(0)} \ip{\zeta, x} + \frac{1}{2}\normE{\zeta}^2
\end{align*}
where the first equality holds because sublinearity implies $\sigma(y) = \sup_{\zeta \in \partial \sigma(0)}\ip{\zeta, y}$ for all $y \in \varspace$. 
\end{proof}
\begin{lemma}\label{lem: polyhedralExpressionOfCore}
    For any convex cone $K$, $\core{K} = \set{x \in \varspace \mid \ip{\zeta, x} \leq -1, \zeta \in N_K(0), \normE{\zeta} = 1}$.
\end{lemma}
\begin{proof}
    Let $x \in \core{K}$. Then for any $\zeta \in N_K(0)$ with $\normE{\zeta} = 1$, we have
    \begin{align*}
        \ip{\zeta, x} = -1 + \ip{\zeta, x + \zeta} \leq -1
    \end{align*}
    where the first equality follows because $\ip{\zeta, \zeta} = 1$ and the inequality follows because $x \in \core{K}$ implies $x + \zeta \in K$ since $\normE{\zeta} \leq 1$.
    Now suppose $\sup_{\zeta \in N_K(0), \normE{\zeta}=1} \ip{\zeta, x} \leq -1$. Then, for any $u$ with $\normE{u} \leq 1$, we have
    \begin{align*}
        \sup_{\zeta \in N_K(0), \normE{\zeta}=1} \ip{\zeta, x + u} \leq \normE{u} + \sup_{\zeta \in N_K(0), \normE{\zeta}=1} \ip{\zeta, x} \leq 0
    \end{align*}
    where the first inequality follows from the Cauchy-Schwarz inequality and the last follows by assumption. Since $K$ is a closed convex cone, this implies $x + u \in K$. Since $u$ was arbitrary, it follows that $x \in \core{K}$.
\end{proof}

From Lemma~\ref{lem: polyhedralExpressionOfCore}, it is immediate that the conic center $x_K$, defined above as the projection onto $\core{K}$, exists and is unique. Similarly, from Lemma~\ref{lem: coreSigmaIsPriceEpigraph}, it follows that $\underset{(x, r) \in C_\sigma}{\min \: } r + \frac{1}{2}\normE{x}^2$ has a unique solution given by
\begin{equation}\label{focusFormula}
    \cCenter{\sigma} = \argmin_{x \in \varspace}\price{\sigma}(x) + \frac{1}{2}\normE{x}^2 \quad \text{and} \quad r_\sigma = \price{\sigma}(\cCenter{\sigma})
\end{equation}
since
$$ (\cCenter{\sigma}, r_\sigma) = \underset{(x, r) \in \epi \price{\sigma}}{\argmin \: } r + \frac{1}{2}\normE{x}^2 = \underset{(x,r), r=\price{\sigma}(x)}{\argmin \: } \price{\sigma}(x) + \frac{1}{2}\normE{x}^2 \ . $$ 
Noting $\price{\sigma}(x) + \frac{1}{2}\normE{x}^2$ is 1-strongly convex, this minimization has a unique minimizer.

Lastly, we note that our conic width $\cGap{K}$ is closely related to the conic width of Xiong and Freund~\cite[Section 1.3]{xiong2024rolelevelsetgeometryperformance}, which is an important quantity in the analysis of conic programming algorithms (see~\cite{xiong2024rolelevelsetgeometryperformance, Freund1999}). In fact, it can easily be shown that their conic width is equal to $1 /(\cGap{K} + 1) = 1/\normE{\cCenter{K}}$.

\subsection{Characterization of All Optimal Smoothings}
We find that the set of optimal smoothings of a given sublinear $\sigma$ always takes the structured form of all $\beta$-smooth convex functions lying pointwise between two extremal smoothings. In particular, we consider the following minimal and maximal $1$-smoothings:
\begin{align}
\begin{matrix}
    \sGen{f}{\sigma} = \left[\infConv{\left(r_\sigma + \sigma(\cdot - \cCenter{\sigma})\right)}{\frac{1}{2}\normE{\cdot}^2}\right] - \cGap{\sigma}/2, & & & & \sGen{F}{\sigma}  = \left[\infConv{\price{\sigma}}{\frac{1}{2}\normE{\cdot}^2}\right] - \cGap{\sigma}/2
\end{matrix}
\end{align}
 where $\infConv{f}{g}$ denotes the infimal convolution between two functions $f:\varspace \to \RR$ and $g : \varspace \to \RR$, i.e.
\begin{equation*}
    (\infConv{f}{g})(x) = \inf_{y \in \varspace} f(y) + g(x - y) .
\end{equation*}
Similarly, for a closed convex cone $K$, to characterize all smoothings, we consider the following minimal and maximal $1$-smoothings:
\begin{align}
    \begin{matrix}
    \sGen{s}{K} = \frac{1}{1 + \frac{\cGap{K}}{2}}\left[(\cCenter{K} + K) + B(0, 1 + \frac{\cGap{K}}{2})\right], & & & & \sGen{S}{K}= \frac{1}{1 + \frac{\cGap{K}}{2}}\left[\core{K} + B(0, 1 + \frac{\cGap{K}}{2})\right] \ . 
\end{matrix}
\end{align}
The following theorems establish for functions and sets, respectively, that the smoothings above provide the smallest and largest optimal smoothings of $\sigma$ and $K$ and suffice to characterize the set of all optimal smoothings. We defer proofs of these main results to Section~\ref{sec: conesProofs}. 
\begin{theorem}\label{thm: optimalSigmaSmoothing}
    A sublinear function $\sigma$ is $\lambda$-smoothable if and only if $\lambda \geq \cGap{\sigma}/2$. In particular, for any target smoothness level $\beta > 0$, a $\beta$-smooth convex $f$ is an optimal $\beta$-smoothing of $\sigma$ if and only if $\scaleOp{\sGen{F}{\sigma}}{1/\beta} \leq f \leq \scaleOp{\sGen{f}{\sigma}}{1/\beta}$.
\end{theorem}
\begin{theorem}\label{thm: optimalConeSmoothing}
    A convex cone $K \subset \varspace$ with a nonempty interior is $\lambda$-smoothable if and only if $\lambda \geq \frac{\cGap{K}}{2 + \cGap{K}}$. In particular, for any target smoothness level $\beta > 0$, a $\beta$-smooth convex $S$ is an optimal $\beta$-smoothing of $K$ if and only if $\scaleOp{\sGen{s}{K}}{1/\beta} \subseteq S \subseteq \scaleOp{\sGen{S}{K}}{1/\beta}$.
\end{theorem}

\subsubsection{Characterizations for Optimal Inner and Outer Smoothings}
\begin{definition}
    A convex function $f: \varspace \to \RR$ is an optimal inner $\beta$-smoothing of $g:\varspace \to \RR$ if $f$ is $\beta$-smooth, $f \geq g$, and $\dist{g, f} \leq \dist{g, \Tilde{f}}$ for all $\beta$-smooth, convex $\Tilde{f}: \varspace \to \RR$ satisfying $\Tilde{f} \geq g$. Analogously, a $\beta$-smooth convex $f \leq g$ is an optimal outer $\beta$-smoothing of $g$ if $\dist{g, f} \leq \dist{g, \Tilde{f}}$ for all $\beta$-smooth, convex $\Tilde{f} \leq g$. 
\end{definition}
\begin{definition}
    A closed convex set $S$ is an optimal inner $\beta$-smoothing of another set $C$ if $S$ is $\beta$-smooth, $S \subseteq C$, and  $\dist{C, S} \leq \dist{C, \Tilde{S}}$ for all $\beta$-smooth, convex $\Tilde{S} \subseteq \varspace$ satisfying $\Tilde{S} \subseteq C$. Analogously, a $\beta$-smooth, closed convex $S \supseteq C$ is an optimal outer $\beta$-smoothing of $C$ if $\dist{C, S} \leq \dist{C, \Tilde{S}}$ for all $\beta$-smooth, closed convex $\Tilde{S} \supseteq C$.
\end{definition}

For any sublinear $\sigma$, we find that all optimal inner $1$-smoothings of $\sigma$ lie between the following functions, defined as infimal convolutions
\begin{align}
\begin{matrix}
    \sIn{f}{\sigma} = \infConv{\left(r_\sigma + \sigma(\cdot - \cCenter{\sigma})\right)}{\frac{1}{2}\normE{\cdot}^2}, & & & &\fBgInner{\sigma} = \infConv{\price{\sigma}}{\frac{1}{2}\normE{\cdot}^2} .
\end{matrix}
\end{align}
For any convex cone $K$, we find that all optimal inner $1$-smoothings lie between the following sets, defined as Minkowski sums
\begin{align}
    \begin{matrix}
    \sIn{s}{K} = (x_K + K) + B(0, 1), & & & & \bgInner{K} = \core{K} + B(0, 1) . \label{coneInnerSmoothingFormula}
\end{matrix}
\end{align}
Theorems~\ref{thm: optimalSigmaInnerSmoothing}~and~\ref{thm: optimalConeInnerSmoothing} formalize this and are proven in Section~\ref{sec: conesProofs}.
\begin{theorem}\label{thm: optimalSigmaInnerSmoothing}
    Every sublinear function $\sigma$ is inner $\lambda$-smoothable if and only if $\lambda \geq \cGap{\sigma}$. In particular, for any target smoothness level $\beta > 0$, a $\beta$-smooth convex $f$ is an optimal inner $\beta$-smoothing of $\sigma$ if and only if $\scaleOp{\fBgInner{\sigma}}{1/\beta} \leq f \leq \scaleOp{\sIn{f}{\sigma}}{1/\beta}$.
\end{theorem}
\begin{theorem}\label{thm: optimalConeInnerSmoothing}
    Every convex cone $K \subset \varspace$ with a nonempty interior is inner $\lambda$-smoothable if and only if $\lambda \geq \cGap{K}$. In particular, for any target smoothness level $\beta > 0$, a $\beta$-smooth convex $S$ is an optimal inner $\beta$-smoothing of $K$ if and only if $\scaleOp{\sIn{s}{K}}{1/\beta} \subseteq S \subseteq \scaleOp{\bgInner{K}}{1/\beta}$.
\end{theorem}

Similarly, the set of optimal outer $1$-smoothings of $\sigma$ lies between the following functions
\begin{align}
\begin{matrix}
    \sOut{f}{\sigma} = \left[\infConv{\left(r_\sigma + \sigma(\cdot - \cCenter{\sigma})\right)}{\frac{1}{2}\normE{\cdot}^2}\right] - \cGap{\sigma}, & & & &\fBgOuter{\sigma} = \left[\infConv{\price{\sigma}}{\frac{1}{2}\normE{\cdot}^2}\right] - \cGap{\sigma} ,
\end{matrix}
\end{align}
and the set of optimal outer $1$-smoothings of $K$ lies between the following sets
\begin{align}
    \begin{matrix}
    \smOuter{K} = \frac{1}{1 + \cGap{K}}\left[(\cCenter{K} + K) + B(0, 1 + \cGap{K})\right], & & & & \bgOuter{K} = \frac{1}{1 + \cGap{K}}\left[\core{K} + B(0, 1 + \cGap{K})\right].
\end{matrix}
\end{align}
Theorems~\ref{thm: optimalSigmaOuterSmoothing}~and~\ref{thm: optimalConeOuterSmoothing} formalize this and are proven in Section~\ref{sec: conesProofs}.
\begin{theorem}\label{thm: optimalSigmaOuterSmoothing}
    Every sublinear function $\sigma$ is outer $\lambda$-smoothable if and only if $\lambda \geq \cGap{\sigma}$. In particular, for any target smoothness level $\beta > 0$, a $\beta$-smooth convex $f$ is an optimal outer $\beta$-smoothing if and only if $\scaleOp{\fBgOuter{\sigma}}{1/\beta} \leq f \leq \scaleOp{\sOut{f}{\sigma}}{1/\beta}$.
\end{theorem}
\begin{theorem}\label{thm: optimalConeOuterSmoothing}
    Every convex cone $K \subset \varspace$ with a nonempty interior is outer $\lambda$-smoothable if and only if $\lambda \geq \frac{\cGap{K}}{1 + \cGap{K}}$. In particular, for any target smoothness level $\beta > 0$, a $\beta$-smooth convex $S$ is an optimal outer $\beta$-smoothing of $K$ if and only if $\scaleOp{\smOuter{K}}{1/\beta} \subseteq S \subseteq \scaleOp{\bgOuter{K}}{1/\beta}$.
\end{theorem}

\subsubsection{Two Remarks on Optimal Smoothings}
First, we explicitly highlight the parallels between the optimal function and set smoothings.
Recall that (see \cite[Excercise 1.28]{rockafellar2009}) strong convexity of $f + \frac{1}{2}\normE{\cdot}^2$ guarantees that
\begin{equation}\label{epiAddition}
        \epi \infConv{f}{\frac{1}{2}\normE{\cdot}^2} = \epi f + \epi \frac{1}{2}\normE{\cdot}^2 \quad \text{for all convex } f: \varspace \to \RR .
\end{equation}
Noting that $\epi \sigma$ is a convex cone, \eqref{epiAddition} reveals the following symmetry between the optimal inner smoothings of sublinear functions and convex cones. Observe that
\begin{align*}
    \epi \sIn{f}{\sigma} & = ((\cCenter{\sigma}, r_\sigma) + \epi \sigma) + \epi \frac{1}{2}\normE{\cdot}^2, \quad \quad \quad \quad \quad \: \epi \fBgInner{\sigma} = \core{\sigma} \: + \: \epi \frac{1}{2}\normE{\cdot}^2,\\
    \sIn{s}{K} &= \qquad (x_K + K) \quad  \: + B(0, 1), \qquad \qquad \qquad \qquad \bgInner{K} = \core{K} + B(0, 1) .
\end{align*}
The other optimal sets are generated from these inner smoothings. For functions, a simple translation gives the other optimal smoothings, as
\begin{align}\label{eq: reformulationOfOptFunctions}
\begin{matrix}
    \sGen{f}{\sigma} = \sIn{f}{\sigma} - \cGap{\sigma}/2, & & & & \sGen{F}{\sigma}  = \fBgInner{\sigma} - \cGap{\sigma}/2, \\
    & & & & \\
    \sOut{f}{\sigma} = \sIn{f}{\sigma} - \cGap{\sigma}, & & & &\fBgOuter{\sigma} = \fBgInner{\sigma} - \cGap{\sigma} .
\end{matrix}
\end{align}
The parallel transformation for sets is a Minkowski addition together with a rescaling, giving
\begin{align}\label{eq: reformulationOfOptSets}
    \begin{matrix}
    \sGen{s}{K} = \frac{1}{1 + \frac{\cGap{K}}{2}}\left[\sIn{s}{K} + B(0, \frac{\cGap{K}}{2})\right], & & & & \sGen{S}{K}= \frac{1}{1 + \frac{\cGap{K}}{2}}\left[\bgInner{K} + B(0, \frac{\cGap{K}}{2})\right], \\
    & & & & \\
    \smOuter{K} = \frac{1}{1 + \cGap{K}}\left[\sIn{s}{K} + B(0, \cGap{K})\right], & & & & \bgOuter{K} = \frac{1}{1 + \cGap{K}}\left[\bgInner{K} + B(0, \cGap{K})\right].
\end{matrix}
\end{align}
The reformulations in \eqref{eq: reformulationOfOptFunctions} and \eqref{eq: reformulationOfOptSets} are easy to verify from definitions.

Second, we note that our theory provides an exact characterization of when a unique optimal smoothing exists.
For both the functional and conic setting, all the smoothings are unique if and only if the core is the translation of the original with the origin placed at the center:
\begin{align}
    & \sigma \text{ has unique smoothings} \iff \price{\sigma} = r_\sigma + \sigma(\cdot - \cCenter{\sigma}) \label{uniquenessOfSigmaSmoothings}, \\
    & K \text{ has unique smoothings} \iff \core{K} = \cCenter{K} + K. \label{uniquenessOfConeSmoothings}
\end{align}

\subsection{Some Families of Optimal Smoothings}\label{subsec: examples}
Before proving our main theorems in the following subsection, here we apply them to several important, well-studied nonsmooth functions and sets. First, applying our theory to sublinear functions, we present the optimal smoothings of the ReLU function, an optimal smoothing of any norm based on its Moreau envelope (not necessarily unique), and optimal, computationally tractable smoothings of the maximum and maximum eigenvalue functions. Second, applying our theory to convex cones, we present the unique optimal smoothings of the nonnegative, second-order, and semidefinite cone and numerically identify the set of all optimal smoothings of the exponential cone.

For these examples, the following lemma is useful, giving a formula for the infimal convolution with $\frac{1}{2}\normE{\cdot}^2$ (i.e., the Moreau envelope) of any sublinear function. The minimal optimal smoothings $\sGen{f}{\sigma}, \sIn{f}{\sigma}, \sOut{f}{\sigma}$ are all translations of such functions. (Note the maximal optimal smoothings $\sGen{F}{\sigma}, \sIn{F}{\sigma}, \sOut{F}{\sigma}$ generally are not translations of Moreau envelopes of sublinear functions.)
\begin{lemma}\label{lem: MoreauOfSublinear}
    For any sublinear function $\sigma$ and $x \in \varspace$,
    \begin{equation}\label{MoreauOfSigma}
        (\infConv{\sigma}{\frac{1}{2}\normE{\cdot}^2})(x) = \frac{1}{2}\normE{x}^2 - \frac{1}{2}\dist{x, \partial \sigma(0)}^2 \ . 
    \end{equation}
\end{lemma}
\begin{proof}
    First, we assert that 
    \begin{equation}\label{proxOfSigma}
        \argmin_{y \in \varspace} \sigma(y) + \frac{1}{2}\normE{x - y}^2 = x - P_{\partial \sigma(0)}(x) .
    \end{equation}
    Let $x^+ = x - P_{\partial \sigma(0)}(x)$. It suffices to show that $x - x^+ \in \partial \sigma(x^+)$ as this is the necessary and sufficient condition to attain the minimum in \eqref{proxOfSigma}. For any $\zeta \in \partial \sigma(0)$, we have
    \begin{align*}
        \ip{x^+, \zeta} & =  \ip{x - P_{\partial \sigma(0)}(x), \zeta} \leq \ip{x - P_{\partial \sigma(0)}(x), P_{\partial \sigma(0)}(x)} = \ip{x^+, x - x^+}
    \end{align*}
    where the inequality follows by normality conditions. Since $\sigma$ is sublinear, this implies $x - x^+ \in \partial \sigma(x^+)$ and hence \eqref{proxOfSigma}.
    Then the claimed general formula for the Moreau envelope of $\sigma$ follows as
    \begin{align*}
        (\infConv{\sigma}{\frac{1}{2}\normE{\cdot}^2})(x)
        & = \inf_{y \in \varspace} \sigma(y) + \frac{1}{2}\normE{x - y}^2\\
        & = \sigma\left(x - P_{\partial \sigma(0)}(x)\right) + \frac{1}{2}\normE{P_{\partial \sigma(0)}(x)}^2 \\
        & = \ip{x  - P_{\partial \sigma(0)}(x), P_{\partial \sigma(0)}(x)} + \frac{1}{2}\normE{P_{\partial \sigma(0)}(x)}^2 \\
        & = \frac{1}{2}\normE{x}^2 - \frac{1}{2}\normE{x - P_{\partial \sigma(0)}(x)}^2
    \end{align*}
    where the second equality follows by \eqref{proxOfSigma}, the third because $P_{\partial \sigma(0)}(x) \in \partial \sigma (x - P_{\partial \sigma(0)}(x))$ by normality conditions, and the last is just completing the square. Since $\dist{x, \partial \sigma(0)} = \normE{x - P_{\partial \sigma(0)}(x)}$, the claimed formula follows.
\end{proof}

\subsubsection{The Optimal Smoothings of the ReLU Function}
    As a simple first example of calculating optimal smoothing with our theory, consider the ReLU (rectified linear unit) function denoted $\sigma(x) \defeq \max\set{0, x}$ defined with $x\in\varspace=\RR$. A wide range of ad hoc smoothings of this function have been used in the machine learning literature~\cite{barron2021squareplussoftpluslike, ramachandran2017searchingactivationfunctions, pmlr-v15-glorot11a,clevert2016fastaccuratedeepnetwork}. Our theory provides principled alternatives. Noting that $\partial \sigma(0) = \mathrm{conv}\{0,1\}$, the functional core can be computed as
    \begin{align*}
        \price{\sigma}(x) = \max_{\zeta \in \partial \sigma(0)}\zeta x + \frac{1}{2}\zeta^2 = \max_{\zeta \in \{0,1\}}\zeta x + \frac{1}{2}\zeta^2 = \max\set{0, x + \frac{1}{2}} = \sigma(x + 1/2) .
    \end{align*}
    The functional center is given by $(\cCenter{\sigma},r_\sigma) = (-1/2, 0)$ and the functional width is $\cGap{\sigma} = 1/8$.
    Observing that $\price{\sigma} = r_\sigma + \sigma(\cdot - \cCenter{\sigma})$, \eqref{uniquenessOfSigmaSmoothings} ensures that the ReLU function has a unique optimal smoothing. In particular, from Theorem~\ref{thm: optimalSigmaSmoothing}, calculating an infimal convolution of $\frac{1}{2}\normE{\cdot}^2$ with $\price{\sigma}$ gives the unique optimal $1$-smoothing as
    \begin{equation}
        \sGen{f}{\sigma}(x) = 
        \begin{cases}
            - \frac{1}{16} & \text{if } x < -1/2 \\
            \frac{1}{2}(x + \frac{1}{2})^2 - \frac{1}{16} & \text{if } -1/2 \leq x \leq 1/2 \\
            x - \frac{1}{16} & \text{otherwise.}
        \end{cases}
    \end{equation}
    The error of approximation is $\dist{\sGen{f}{\sigma},\sigma} = \frac{1}{2}\cGap{\sigma} = 1/16$. More generally, $\frac{1}{\beta}\sGen{f}{\sigma}(\beta x)$ gives the optimal $\beta$-smoothing for any $\beta > 0$, with approximation error of $\frac{1}{16\beta}$.
    
\subsubsection{An Optimal Smoothing for any Norm}
For any norm $\sigma = \norm{\cdot}$ (not necessarily the Euclidean norm), we find that the function center is very structured. Namely, $\cCenter{\sigma} = 0$ and $\cGap{\sigma} = \price{\sigma}(0)$. From this, one can conclude that the Moreau envelope of the norm, $\infConv{\sigma}{\frac{\beta}{2}\normE{\cdot}^2}$, provides an optimal outer $\beta$-smoothing. From this optimal outer smoothing, optimal smoothings in general and inner smoothings are given by appropriately adding constants to this envelope via~\eqref{eq: reformulationOfOptFunctions}.
\begin{proposition}\label{prop: MoreauOfNormIsOptimal}
    For any norm $\sigma$ and $\beta>0$, the Moreau Envelope $\infConv{\sigma}{\frac{\beta}{2}\normE{\cdot}^2}$ is an optimal outer $\beta$-smoothing of $\sigma$. An optimal general smoothing and inner smoothing is given by $(\infConv{\sigma}{\frac{\beta}{2}\normE{\cdot}^2}) + \beta\cGap{\sigma}/2$ and $(\infConv{\sigma}{\frac{\beta}{2}\normE{\cdot}^2}) + \beta\cGap{\sigma}$, respectively, where $\cGap{\sigma} = \max_{\zeta \in \partial \sigma(0)} \frac{1}{2}\normE{\zeta}^2$. 
\end{proposition}
\begin{proof}
    We will show that $\sOut{f}{\sigma} = \infConv{\sigma}{\frac{1}{2}\normE{\cdot}^2}$. Noting that $\infConv{\sigma}{\frac{\beta}{2}\normE{\cdot}^2} = \scaleOp{\infConv{\sigma}{\frac{1}{2}\normE{\cdot}^2}}{1/\beta}$ by positive homogeneity of $\sigma$, this will prove optimality of $\infConv{\sigma}{\frac{\beta}{2}\normE{\cdot}^2}$, by Theorem~\ref{thm: optimalSigmaOuterSmoothing}. Optimal general and inner smoothings will then follow by~\eqref{eq: reformulationOfOptFunctions}. Now, observe that
    \begin{align*}
       \sOut{f}{\sigma} = \left[\infConv{r_{\sigma} + \sigma(\cdot - \cCenter{\sigma})}{\frac{1}{2}\normE{\cdot}^2}\right] - \cGap{\sigma} = \left[\infConv{\sigma(\cdot - \cCenter{\sigma})}{\frac{1}{2}\normE{\cdot}^2}\right] - \frac{1}{2}\normE{\cCenter{\sigma}}^2 
    \end{align*}
    where the first equality is by definition and the second follows because $\infConv{(c + f)}{g} = r + \left(\infConv{f}{g}\right)$ for any $r \in \RR$ and two functions $f$ and $g$, and $\cGap{\sigma} = r_\sigma + \frac{1}{2}\normE{\cCenter{\sigma}}^2 $ by definition. Therefore, $\sOut{f}{\sigma} = \infConv{\sigma}{\frac{1}{2}\normE{\cdot}^2}$ follows immediately once we show that $\cCenter{\sigma} = 0$. Further, from the identity $\cCenter{\sigma} = 0$, we will get that $\cGap{\sigma} = r_\sigma = \price{\sigma}(0) = \max_{\zeta \in \partial \sigma(0)} \frac{1}{2}\normE{\zeta}^2$ by Equations~\eqref{priceFormula}~and~\eqref{focusFormula}. Therefore proving that $\cCenter{\sigma} = 0$ will conclude the proof.
    
    Since $\cCenter{\sigma} = \argmin_{x \in \varspace}\price{\sigma}(x) + \frac{1}{2}\normE{x}^2$ by Equation~\ref{focusFormula}, the first order condition $0 \in \partial \price{\sigma}(0)$ is sufficient to prove $\cCenter{\sigma} = 0$. Observe that (i) $\argmax_{\zeta \in \partial \sigma(0)}\frac{1}{2}\normE{\zeta}^2 \subseteq \partial \price{\sigma}(0)$ by \eqref{priceFormula} and (ii) $\partial \sigma(0) = - \partial \sigma(0)$ because $\sigma$ is a norm. Combining (i) and (ii) gives $0 = (\hat{\zeta} - \hat{\zeta})/2 \in \partial \price{\sigma}(0)$ whenever $\hat{\zeta} \in \argmax_{\zeta \in \partial \sigma(0)}\frac{1}{2}\normE{\zeta}^2$, by convexity of the subdifferential. Since $\partial \sigma(0)$ is compact, it follows $\argmax_{\zeta \in \partial \sigma(0)}\frac{1}{2}\normE{\zeta}^2$ is nonempty. Therefore $0 \in \partial \price{\sigma}(0)$ and the proof is complete.
\end{proof}

Note that this result does not ensure that the Moreau envelope is the unique optimal outer smoothing. Depending on the structure of the norm under consideration, the optimal smoothing may or may not be unique. Below, we explicitly calculate two examples (one of each type) via their functional core, center, and width. 

First consider the one-norm $\sigma (x) = \norm{x}_1 \defeq \sum_{i=1}^d\abs{x_i}$ defined in $\varspace = \RR^d$. Noting that $\partial \sigma(0) = \mathrm{conv}\set{(\pm 1, \dots \pm 1)}$, we can calculate the functional core as
\begin{align*}
    \price{\sigma}(x) = \max_{\zeta \in \partial \sigma(0)} \ip{\zeta, x} + \frac{1}{2}\norm{\zeta}_2^2 = \max_{\zeta \in \set{(\pm 1, \dots \pm 1)}} \ip{\zeta, x} + \frac{1}{2}\norm{\zeta}_2^2 = \sigma(x) + \frac{d}{2}.
\end{align*}
Moreover, the functional center is given by $\cCenter{\sigma} = (0, \dots, 0)$ and $ r_\sigma = \frac{d}{2}$. Hence $\price{\sigma} = r_\sigma + \sigma(\cdot - \cCenter{\sigma})$ and so \eqref{uniquenessOfSigmaSmoothings} ensures the one-norm has a unique optimal smoothing. Proposition~\ref{prop: MoreauOfNormIsOptimal} then gives the unique optimal outer $1$-smoothing of
    
$$
\sOut{f}{\sigma} = \infConv{\norm{\cdot}_1}{\frac{1}{2}\norm{\cdot}_2^2} = \sum_{i=1}^d H_1(x_i) \qquad \text{where} \quad H_1(x_i) = 
\begin{cases} 
    \frac{1}{2}\abs{x_i}^2 & \text{if } \abs{x_i} \leq 1 \\ 
        \abs{x_i} - \frac{1}{2} & \text{otherwise.}
\end{cases}    
$$
The error of approximation is $\dist{\sOut{f}{\sigma}, \sigma} = \frac{1}{2}\cGap{\sigma} = \frac{d}{2}$. The optimal general and inner $1$-smoothings are then given by $\sum_{i=1}^d H_1(x_i) + \frac{d}{4}$ and $\sum_{i=1}^d H_1(x_i) + \frac{d}{2}$, respectively.

Second, as a less structured example norm, consider the following weighted infinity norm $\sigma(x) = \max\set{|x_1|, 2|x_2|}$ defined in $\varspace = \RR^2$.
Noting that $\partial \sigma(0) = \mathrm{conv}\set{\pm(1,0), \pm(0, 2)}$, one can calculate the functional core as
    \begin{align*}
        \price{\sigma}(x) = \max_{\zeta \in \partial \sigma(0) } \ip{\zeta, x} + \frac{1}{2}\norm{\zeta}_2^2  = \max_{\zeta \in\set{\pm(1,0), \pm(0, 2)}} \ip{\zeta, x} + \frac{1}{2}\norm{\zeta}_2^2 = \max\set{|x_1| + \frac{1}{2}, 2|x_2| + 2}.
    \end{align*}
Moreover, the functional center is given by $\cCenter{\sigma} = (0, 0)$ and $r_\sigma = 2$. Since $\sigma(x) + r_\sigma > \price{\sigma}(x)$ whenever $|x_1| \geq 2$ and $|x_2| = 0$, we note that $\sigma$ does not have unique optimal smoothings, by \eqref{uniquenessOfSigmaSmoothings}. In this case, the largest (in function value) optimal outer $1$-smoothing is given by
\[ 
\sOut{f}{\sigma}(x) = \infConv{\sigma}{\frac{1}{2}\norm{\cdot}_2^2} =
\begin{cases}
\frac{1}{2}\norm{x}_2^2 & \text{if } |x_1| + \frac{1}{2}|x_2| \leq 1 \\
|x_1| - \frac{1}{2} & \text{else if }  2|x_2|-|x_1| \leq - 1 \\
2|x_2| - 2 & \text{else if }  2|x_2| - |x_1| \geq 4\\
\frac{1}{5}\left(4|x_1|+ 2|x_2| - 2 + \frac{1}{2}(2|x_2|-|x_1|)^2\right) & \text{else}
\end{cases} \]
while the smallest optimal outer $1$-smoothing is given by
\[
\sOut{F}{\sigma}(x) =  \infConv{(\price{\sigma} - \cGap{\sigma})}{\frac{1}{2}\norm{\cdot}_2^2} = 
\begin{cases}
\sOut{f}{\sigma}((0, x_2)) & \text{if } |x_1| \leq 3/2  \\
\sOut{f}{\sigma}(x + (3/2, 0)) & \text{if } x_1 < -3/2 \\
\sOut{f}{\sigma}(x - (3/2, 0)) & \text{else}.
\end{cases} 
\]
Hence, a $1$-smooth convex function $f$ is an optimal $1$-smoothing of this weighted infinity norm exactly if it lies between these two functions.

\subsubsection{The Optimal Smoothings of Maximum and Maximum Eigenvalue Functions}
    Consider the max function $\sigma (x) \defeq \max\set{x_1, \dots, x_d}$, defined in $\varspace = \RR^d$. Noting that $\partial \sigma(0) = \mathrm{conv}\set{e_1, \dots, e_d}$, where $e_i$ is the $i$-th standard basis vector, we have
    \begin{align*}
        & \price{\sigma}(x) = \max_{\zeta \in \partial \sigma(0)} \ip{\zeta, x} + \frac{1}{2}\norm{\zeta}_2^2 = \max_{\zeta \in \set{e_1, \dots, e_d}} \ip{\zeta, x} + \frac{1}{2}\norm{\zeta}_2^2 = \sigma(x) + \frac{1}{2}, \\
        & \cCenter{\sigma} = \argmin_{x \in \RR^d}\frac{1}{2}+\sigma(x) + \frac{1}{2}\norm{x}_2^2 = -(1, \dots, 1) / d, \\
        & r_\sigma = \price{\sigma}(\cCenter{\sigma}) = -\frac{1}{d} + \frac{1}{2}.
    \end{align*}
    This implies $\price{\sigma} = r_\sigma + \sigma(\cdot - \cCenter{\sigma})$ and so $\sigma$ has a unique optimal $1$-smoothing. In particular, it is given by\footnote{This can be verified by directly computing $\sGen{F}{\sigma}(x) = \infConv{\price{\sigma}}{\frac{1}{2}\norm{\cdot}_2^2} - \cGap{\sigma}/2 = \infConv{\left(\frac{1}{2} + \sigma\right)}{\frac{1}{2}\norm{\cdot}_2^2} -(1 - 1/d)/4 $ using~\eqref{proxOfSigma} and the formula for projecting on the unit simplex $\partial\sigma(0)= \mathrm{conv}\set{e_1, \dots, e_d}$ given in \cite[Figure 1]{duchi2008efficient}. Further note from this projection formula, this optimal smoothing can be computed in (expected) linear time in $n$~\cite{Condat2016}.}
    \begin{equation}\label{eq: maxFuncOptSmoothing}
        \sGen{f}{\sigma}(x) = \sGen{F}{\sigma}(x) = \alpha + \frac{1}{2}\sum_{i = 1}^J(x_{\pi(i)} - \alpha)^2 + \frac{1}{4}\left(1 + \frac{1}{d}\right) 
    \end{equation}
    where the permutation $\pi$ ensures $x_{\pi(i)}\geq x_{\pi(i+1)}$ for all $i < d$, $\alpha =  \frac{1}{J}(\sum_{i = 1}^Jx_{\pi(i)} - 1)$, and
    $J = \max\set{j \in \{1, \dots, d\} \mid x_{\pi(j)} - \frac{1}{j}(\sum_{i = 1}^{j}x_{\pi(i)} - 1) > 0}$. The error of approximation is $\dist{\sGen{f}{\sigma}, \sigma} = \frac{1}{2}\cGap{\sigma} = \frac{1}{4}\left(1 - \frac{1}{d}\right)$. More generally, $\frac{1}{\beta}\sGen{f}{\sigma}(\beta x)$ gives the optimal $\beta$-smoothing for any $\beta > 0$, with approximation error of $\frac{1}{4\beta}\left(1 - \frac{1}{d}\right)$.

    Entirely similar reasoning and formulas apply to give the unique optimal smoothing of the maximum eigenvalue function $\sigma(A) \defeq \max_{x \in \RR^d, \norm{x}_2 =1}x^TAx$, defined in the space of symmetric $d \times d$ matrices $\varspace = \mathbb{S}^{d\times d}$. Appendix~\ref{app:max-eigen} provides the resulting parallel formulas.

\subsubsection{The Optimal Smoothings of Nonnegative, Second-Order, and Semidefinite Cone}
Perhaps the three most common cones are the nonnegative orthant $\RR^d_+=\{x\in\RR^d \mid x\geq 0\}$, second-order cone $K^\mathrm{SOC}=\{(x,t)\in\RR^{d+1} \mid \norm{x}_2\leq t\}$, and positive semidefinite cone $K^\mathrm{SDP}=\{X\in\mathbb{S}^{d\times d} \mid X \succeq 0\}$. Optimization over these sets corresponds to linear, second-order cone, and semidefinite programming.

Table~\ref{tab:classic-cones} states the conic center and width for each of these cones. In each case, the conic core is exactly $x_K+K$, establishing that they each possess a unique optimal smoothing.

To illustrate the process of calculating these quantities, consider the nonnegative orthant $K = \mathbb{R}^d_+ \defeq \set{x \in \RR^d \mid x_i \geq 0, i=1, \dots, d}$. Then, $N_K(0) = \mathrm{cone}\set{-e_1, \dots, -e_d}$ where $e_i$ denotes the $i$-th standard basis vector. Letting $e =  (1, \dots, 1)$ denote the all ones vector, we have that
    \begin{align*}
        &\core{K} = \set{x \in \varspace \mid \ip{\zeta, x} \leq -1, \norm{\zeta}_2 = 1, \zeta \in N_K(0)} =\set{x \in \RR^d \mid \ip{-e_i, x} \leq -1, i=1, \dots, d} = e + K, \\
        & \cCenter{K} = \argmin_{x \in \core{K}} \norm{x}_2 = e, \\
        & \cGap{K} = \norm{\cCenter{K}}_2 - 1 = \sqrt{d} - 1.
    \end{align*}
    From this, it follows, for example, that $e + \RR^d_+ + B(0,1)$ is the optimal inner $1$-smoothing with approximation error $\sqrt{d} - 1$. Similar processes of calculating give the claimed quantities for the second-order and semidefinite cones.

\begin{table}[t]
    \centering
    \begin{tabular}{|c c c c c|}
        \hline
        $K$ & $C_K$ & $\cCenter{K}$ & $\cGap{K}$ & Unique?\\ 
        \hline
        $\RR^d_+$ & $\set{(x_1, \dots, x_d) \in \RR^d \mid x_i \geq 1, i=1, \dots, d}$ & $(1, \dots, 1)$ & $\sqrt{d} - 1$ & Yes\\
        $K^\mathrm{SOC}$ & $\set{(x, t) \in \varspace \times \RR \mid t-\sqrt{2} \geq \norm{x}_2}$ & $(0, \sqrt{2})$ & $\sqrt{2} - 1$ & Yes\\
        $K^\mathrm{SDP}$ &$\set{x \in \mathbb{S}^{d\times d} \mid x \succeq I_d}$ & $I_d$ & $\sqrt{d} - 1$ & Yes\\
        $K^\mathrm{EXP}$ & See Figure~\ref{fig:expConeSmoothings}(c) & (-1.11957, 1, 1.71471) & 1.27897 & No \\
        \hline
        
    \end{tabular}
    \caption{Classic cones with their conic core, center and width. Values for the exponential cone are numerically computed. The final column indicates whether there is a unique optimal smoothing.}
    \label{tab:classic-cones}
\end{table}

\subsubsection{Numerically Computing All Optimal Smoothings of the Exponential Cone}
Finally, providing a contrast to the previous three particularly nice cones, we consider the exponential cone
$$K^\mathrm{EXP} = \set{(x, y, z) \in \RR^3 \mid z \geq ye^{x/y}, y > 0} \cup \set{(x, y, z) \in \RR^3 \mid x \leq 0, z \geq 0, y = 0}. $$
This cone has found widespread use, providing representations for many applied optimization problems in a standardized conic form. See the numerous references and examples in~\cite{boyd2007gp}.

To the best of our knowledge, the conic core of this cone does not possess a closed form. Instead of an analytic solution, our definitions facilitate a numerical computation of the conic core, center, and width for this set.
A rendering of the numerically computed conic core is given in Figure~\ref{fig:expConeSmoothings}. An immediate observation from this is that the conic core is not a translated copy of the exponential cone, implying by \eqref{uniquenessOfConeSmoothings} that there is not a unique optimal way to smooth the exponential cone. The set of all optimal inner $1$-smoothings is characterized by the minimal and maximal optimal $1$-smoothing $\sIn{s}{K}$ and $\sIn{S}{K}$, respectively, as defined in Equations~\eqref{coneInnerSmoothingFormula}. Numerical renderings of these extremal smoothings are also shown in Figure~\ref{fig:expConeSmoothings}.

We numerically estimate the boundary of the core $\core{K}$ by sampling points $x$ satisfying \[\sup_{\zeta \in N_K(0), \normE{\zeta}=1} \ip{\zeta, x} = -1\] as Lemma~\ref{lem: polyhedralExpressionOfCore} dictates. The supremum is numerically computed using a dense sample from the set $\set{\zeta \in \RR^3 \mid \zeta \in N_K(0), \norm{\zeta}_2=1}$. By computing a projection onto this computed conic core, one can numerically estimate the exponential cone's conic center and width. These numerically computed values are given in Table~\ref{tab:classic-cones}. The smooth sets $\sIn{s}{K}$ and $\sIn{S}{K}$ are obtained by adding a ball to the points in $\cCenter{K} + K$ and $\core{K}$, respectively.

\begin{figure}[t]
    \centering
    \begin{subfigure}[b]{0.32\textwidth}
        \includegraphics[width=\textwidth]{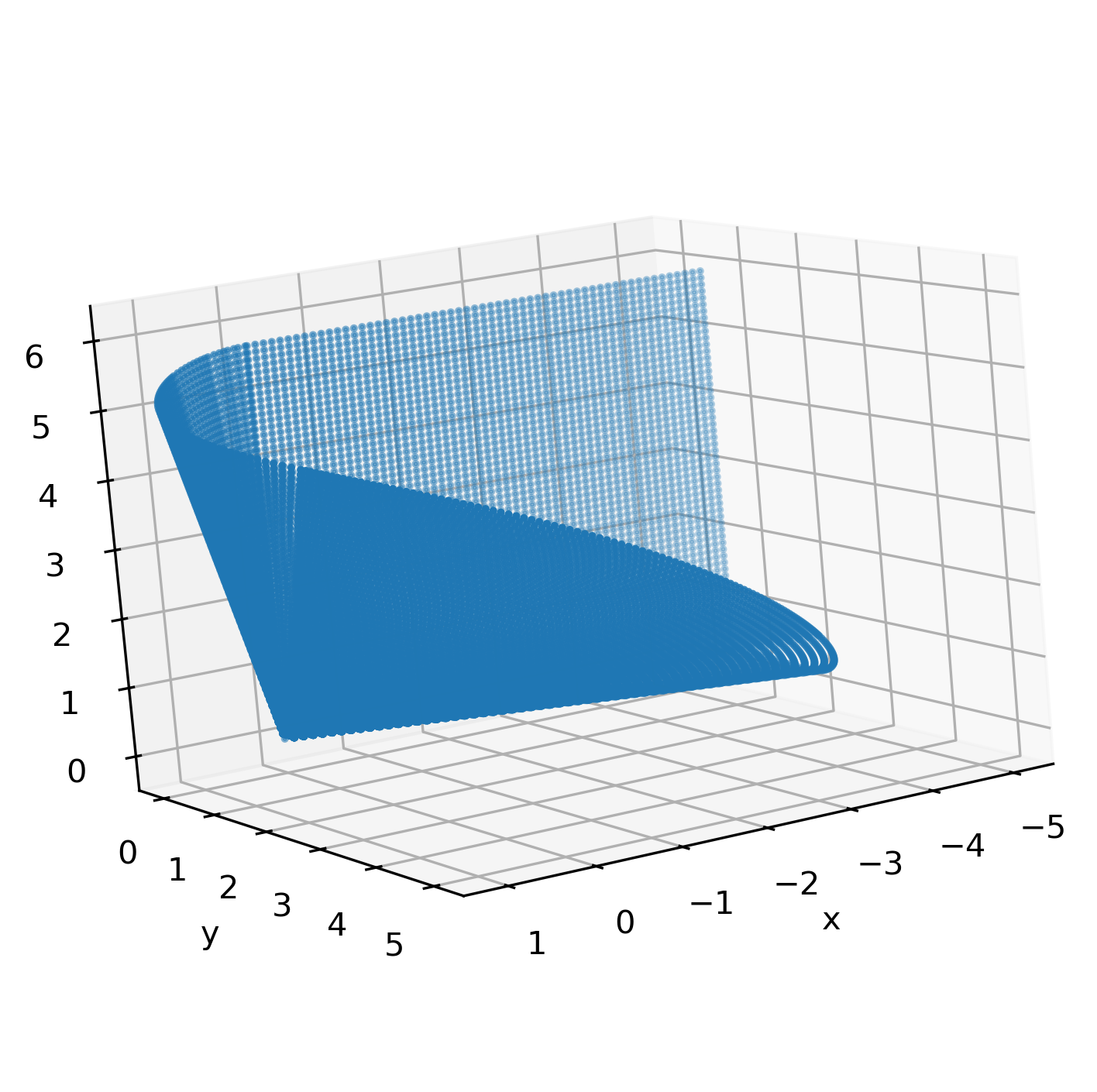}
        \caption{$K$}
    \end{subfigure}
    \hfill
    \begin{subfigure}[b]{0.32\textwidth}
        \includegraphics[width=\textwidth]{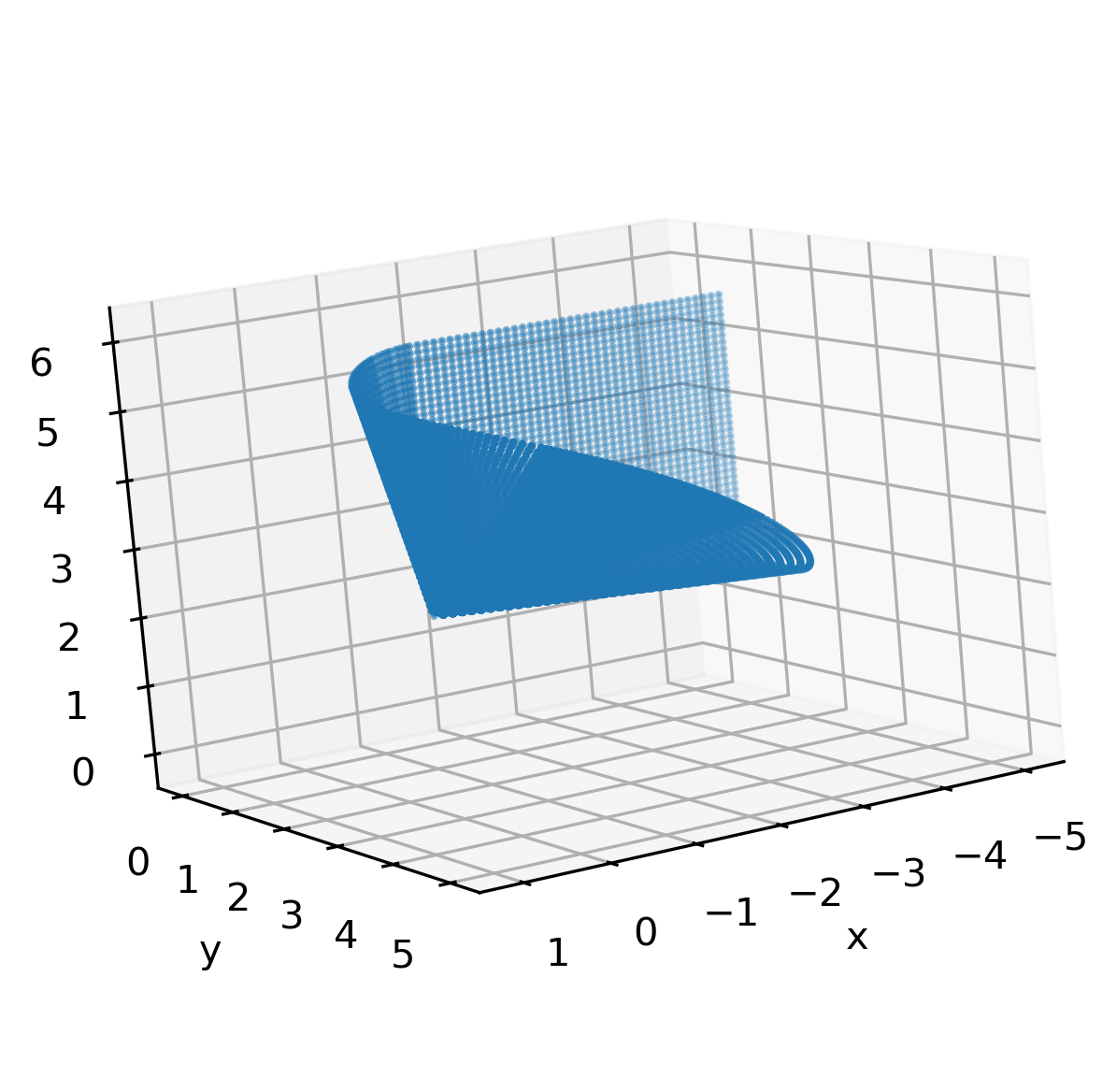}
        \caption{$\cCenter{K}+K$}
    \end{subfigure}
    \hfill
    \begin{subfigure}[b]{0.32\textwidth}
        \includegraphics[width=\textwidth]{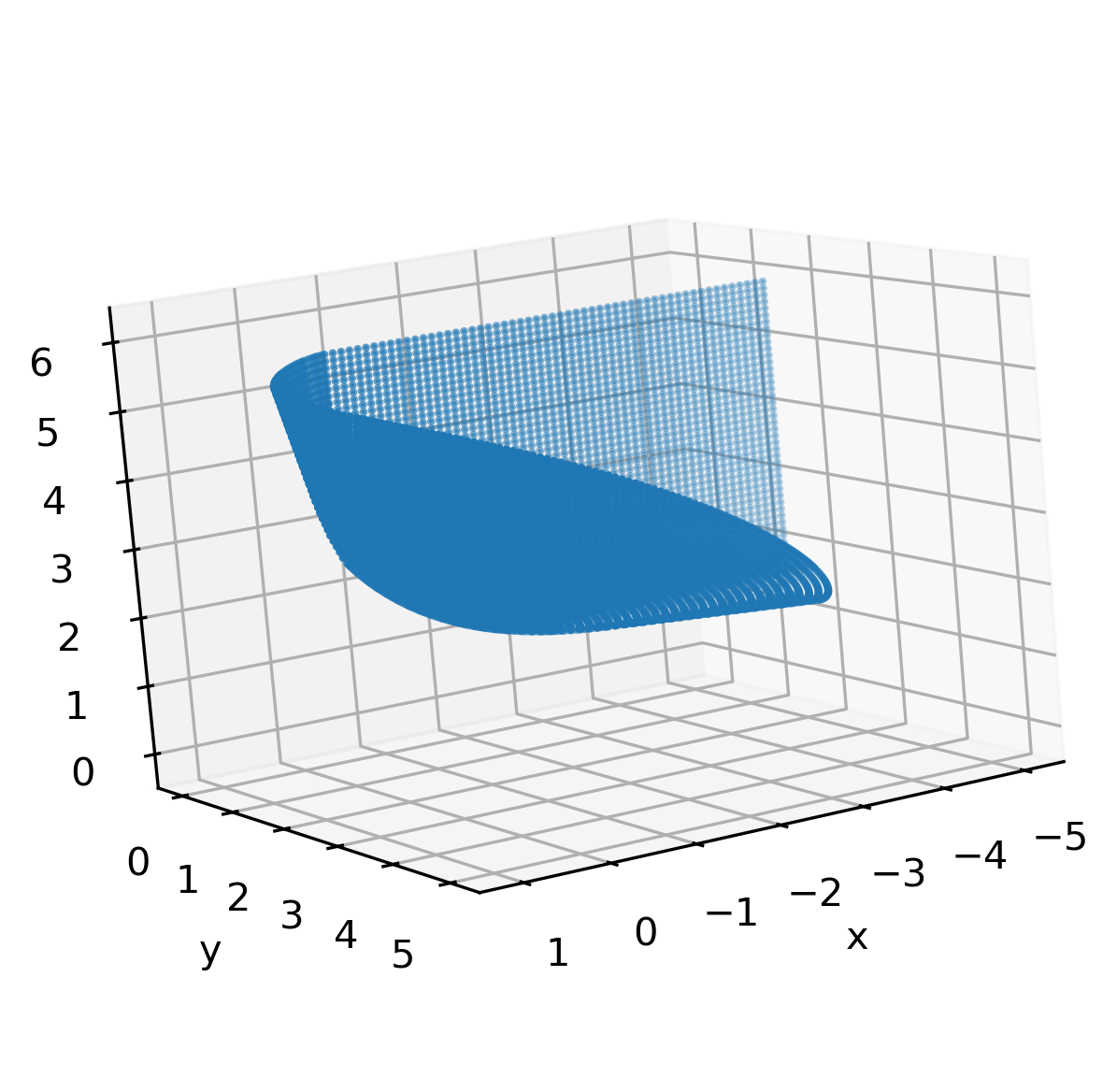}
        \caption{$\core{K}$}
    \end{subfigure}
    \hfill
    \begin{subfigure}[b]{0.32\textwidth}
        \includegraphics[width=\textwidth]{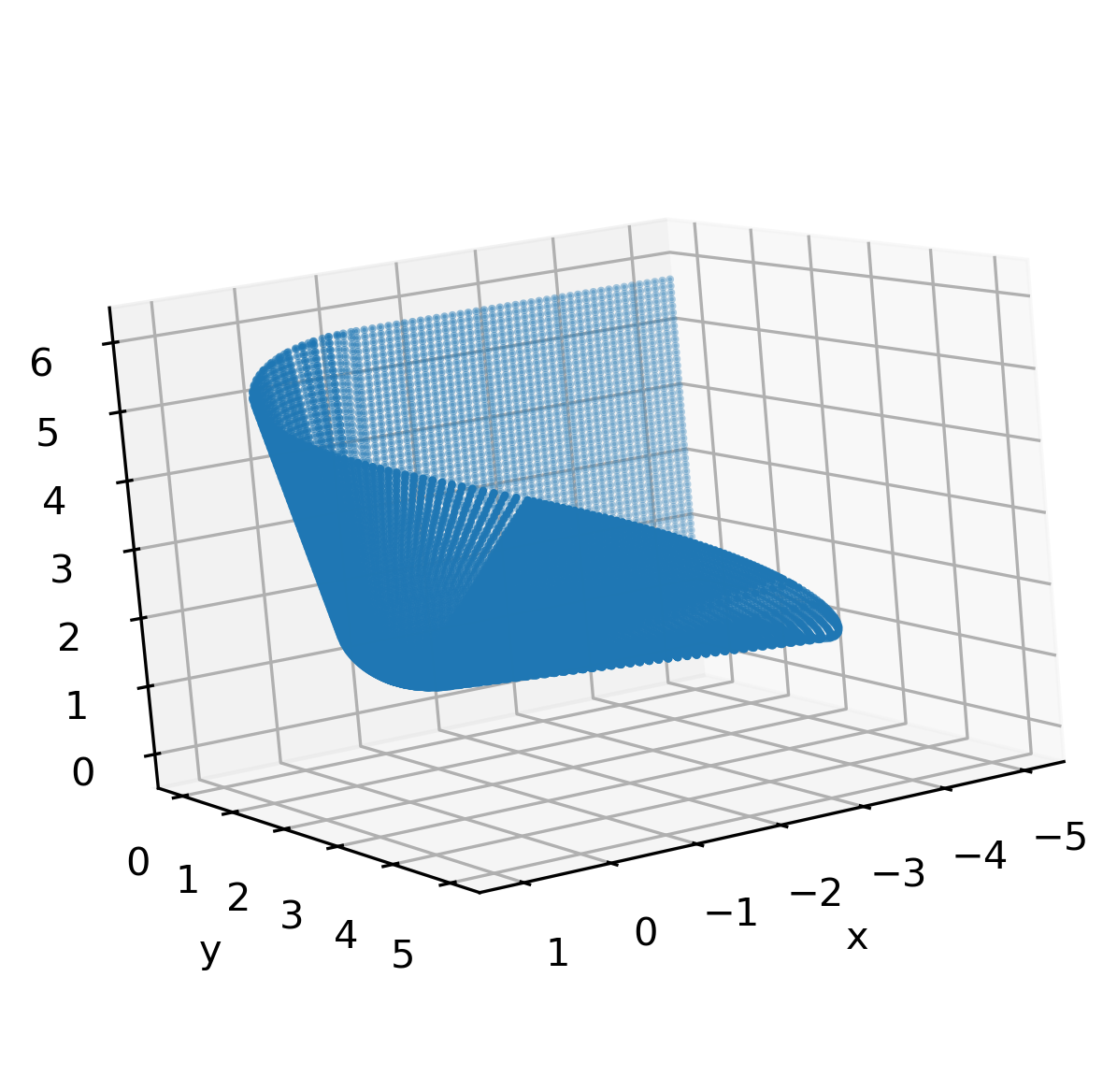}
        \caption{$\sIn{s}{K}$}
    \end{subfigure}
    \begin{subfigure}[b]{0.32\textwidth}
        \includegraphics[width=\textwidth]{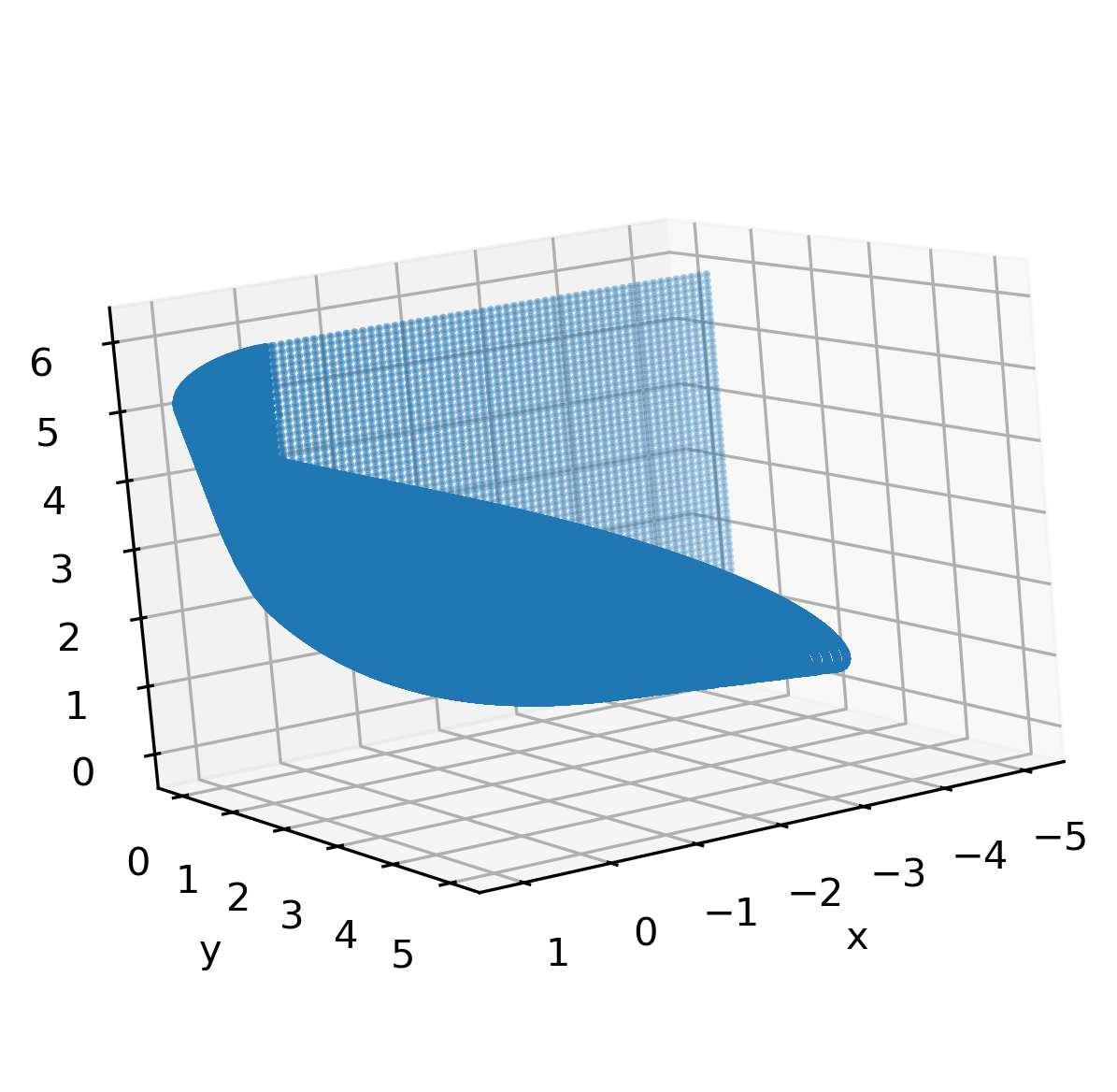}
        \caption{$\bgInner{K}$}
    \end{subfigure}
    \hfill
    \caption{The exponential cone, its translation $x_K+K$, its conic core, its minimal optimal $1$-smoothing, and its maximal optimal $1$-smoothing.}\label{fig:expConeSmoothings}
\end{figure}

%% file: cones_proofs.tex
\subsection{Proofs of  Characterizations of Optimal Smoothings}\label{sec: conesProofs}
We begin with a simple technical lemma (Lemma~\ref{lem: distToConeFixesHorizon}). Then we prove our main theorems for sublinear function and convex cone smoothings, which follow the same pattern: (i) we first prove the result for inner $\beta$-smoothings, and then (ii) we show that characterizations of the optimal general and outer smoothings follow from simple transformations.

\begin{lemma}\label{lem: distToConeFixesHorizon}
    Let $K$ be a nonempty closed convex cone and $C$ be a nonempty closed convex set. Then $K = \hrzn{C}$ whenever $\dist{C, K} < \infty$.
\end{lemma}
\begin{proof}
    Let $r = \dist{C, K}$. We have $\hrzn{C} \subseteq \hrzn{(K + B_r)} = K \subseteq \hrzn{(C + B_r)}  = \hrzn{C} $
    where the inclusions follow because $C \subseteq K + B_r$ and $K \subseteq C + B_r$, while the equalities follow by \cite[Corollary 9.1.2]{cvxAnalysisRockafellar} 
    
\end{proof}

\subsubsection{Proof of Theorem~\ref{thm: optimalSigmaInnerSmoothing}}
Our proof is broken into several lemmas which combine together to immediately yield the theorem. Recall, due to Lemma~\ref{lem: scalingOperation}, it suffices to consider only $\beta = 1$. 
We begin by showing in Lemma~\ref{lem: bigCapSigmaInnerSmoothing} that $\sIn{F}{\sigma}$ is an inner $1$-smoothing that lower bounds any other inner $1$-smoothing of $\sigma$. 

\begin{lemma}\label{lem: bigCapSigmaInnerSmoothing}
    $\sIn{F}{\sigma}$ is an inner $1$-smoothing of $\sigma$ with $\dist{\sIn{F}{\sigma}, \sigma} \geq \cGap{\sigma}$. Further, $\sIn{F}{\sigma} \leq f$ for any convex inner $1$-smoothing $f$.
\end{lemma}
\begin{proof}
    $\sIn{F}{\sigma} = \infConv{\price{\sigma}}{\frac{1}{2}\normE{\cdot}^2}$ is $1$-smooth by \cite[Proposition 12.60]{rockafellar2009}. The inequality $\sIn{F}{\sigma} \geq \sigma$ follows because 
    \begin{align*}
        \epi \sIn{F}{\sigma} = \epi \price{\sigma} + \epi \frac{1}{2}\normE{\cdot }^2 \subseteq \epi \sigma
    \end{align*}
    where the equality follows by \eqref{epiAddition} and the inclusion follows because $\epi \price{\sigma} = \core{\sigma}$ by Lemma~\ref{lem: coreSigmaIsPriceEpigraph}. As a result, we have 
    \[
    \dist{\sIn{F}{\sigma}, \sigma} \geq \sIn{F}{\sigma}(0) - \sigma(0) = \cGap{\sigma}
    \]
    where the inequality holds because $\sIn{F}{\sigma} \geq \sigma$ and the equality follows from Equation~\eqref{focusFormula} and the definition of $\sIn{F}{\sigma}$ and the fact that $\sigma(0) = 0$.
    Now let $f : \varspace \to \RR$ be any $1$-smooth convex function such that $f \geq \sigma$. Then $f = \infConv{g}{\frac{1}{2}\normE{\cdot}^2}$ for some convex function $g$, by \cite[Proposition 12.60]{rockafellar2009}. Note that $\epi g \subseteq \core{\sigma}$ as \eqref{epiAddition} ensures $\epi g + \epi \frac{1}{2}\normE{\cdot}^2 = \epi f \subseteq \epi \sigma$. Since $\core{\sigma} = \epi \price{\sigma}$ by Lemma~\ref{lem: coreSigmaIsPriceEpigraph}, it follows that $g \geq \price{\sigma}$. Hence $f = \infConv{g}{\frac{1}{2}\normE{\cdot}^2}\geq  \infConv{\price{\sigma}}{\frac{1}{2}\normE{\cdot}^2}=\sIn{F}{\sigma}$.
\end{proof}

We next show in Lemma~\ref{lem: lowCapSigmaInnerSmoothing} that $\sIn{f}{\sigma}$ is an inner $1$-smoothing with that upper bounds all inner $1$-smoothings $f$ satisfying $\dist{f, \sigma} \leq \dist{\sIn{f}{\sigma}, \sigma}$. Since any optimal $f$ satisfies $\dist{f, \sigma} \leq \dist{\sIn{f}{\sigma}, \sigma}$ by definition, this proves that $\sIn{f}{\sigma}$ upper bounds optimal smoothings.
\begin{lemma}\label{lem: lowCapSigmaInnerSmoothing}
    $\sIn{f}{\sigma}$ is an inner $1$-smoothing of $\sigma$ with $\dist{\sIn{f}{\sigma}, \sigma} \leq \cGap{\sigma}$. Further, $\sIn{f}{\sigma}\geq f$ for any convex, inner $1$-smoothing $f$ with $\dist{f, \sigma} \leq \cGap{\sigma}$.
\end{lemma}
\begin{proof}
    $\sIn{f}{\sigma}$ is $1$-smooth by \cite[Proposition 12.60]{rockafellar2009}. The inequality $\sIn{f}{\sigma} \geq \sigma$ follows because 
    \begin{align*}
        \epi \sIn{f}{\sigma} = \epi (r_\sigma + \sigma(\cdot - \cCenter{\sigma})) + \epi \frac{1}{2}\normE{\cdot }^2 = (\cCenter{\sigma}, r_\sigma) + \epi \sigma + \epi \frac{1}{2}\normE{\cdot }^2 \subseteq \epi \sigma + \epi \sigma = \epi \sigma
    \end{align*}
    where the first equality follows by \eqref{epiAddition}, the inclusion follows because $(\cCenter{\sigma}, r_\sigma) \in \core{\sigma}$ ensures $(\cCenter{\sigma}, r_\sigma) + \epi \frac{1}{2}\normE{\cdot }^2 \subseteq \epi \sigma$, and the last equality follows because $\epi \sigma$ is a convex cone.
     Therefore
    \begin{align*}
        \dist{\sIn{f}{\sigma}, \sigma} 
        & = \sup_{x \in \varspace} \sIn{f}{\sigma}(x) - \sigma(x) \\
        & = \sup_{x \in \varspace}r_\sigma + \frac{1}{2}\normE{x - \cCenter{\sigma}}^2 - \frac{1}{2}\normE{x - \cCenter{\sigma} - P_{\partial \sigma(0)}(x - \cCenter{\sigma})}^2 - \sigma(x)\\
        & = \sup_{x \in \varspace} r_\sigma + \frac{1}{2} \normE{\cCenter{\sigma}}^2 - \frac{1}{2}\normE{\cCenter{\sigma} + P_{\partial \sigma(0)}(x - \cCenter{\sigma})}^2  + \ip{P_{\partial \sigma(0)}(x - \cCenter{\sigma}), x} - \sigma(x) \\
        & \leq r_\sigma + \frac{1}{2} \normE{\cCenter{\sigma}}^2 \\
        & = \cGap{\sigma}
    \end{align*}
    where the second equality follows by \eqref{MoreauOfSigma}, the inequality follows in part because of the subgradient inequality $\sigma(x) \geq \sigma(0) +  \ip{P_{\partial \sigma(0)}(x - \cCenter{\sigma}), x}$, and the last equality holds by definition of $\cGap{\sigma}$.

    For the second part of the Lemma, note that if $\dist{f, \sigma} \leq \cGap{\sigma}$, then
    \begin{align*}
        \epi \sIn{f}{\sigma} 
        = \epi (r_\sigma + \sigma(\cdot - \cCenter{\sigma})) + \epi \frac{1}{2}\normE{\cdot}^2 
        & = (\cCenter{\sigma}, r_\sigma) + \epi \frac{1}{2}\normE{\cdot}^2 + \epi \sigma \\
        & = (\cCenter{\sigma}, r_\sigma) + \epi \frac{1}{2}\normE{\cdot}^2 + \hrzn{(\epi f)}
    \end{align*}
    
    where the first equality follows by \eqref{epiAddition} and the last equality follows because Lemma~\ref{lem: distToConeFixesHorizon} ensures that $\epi \sigma = \hrzn{(\epi f)}$ as $\dist{\epi f, \epi \sigma} \leq \dist{f, \sigma} \leq  \cGap{\sigma}$. Therefore, to prove that $\epi \sIn{f}{\sigma} \subseteq \epi f$, it suffices to show that $(\cCenter{\sigma}, r_\sigma) + \epi \frac{1}{2}\normE{\cdot}^2 \subseteq \epi f$. We next show that this inclusion holds. Recall that $1$-smoothness implies $f(x) \leq f(0) + \ip{\nabla f(0), x} + \frac{1}{2}\normE{x}^2$ for all $x \in \varspace$ by \eqref{eq: smoothnessQuadBound}. Therefore
    \begin{align}
        \left(- \nabla f(0), f(0)- \frac{1}{2}\normE{\nabla f(0)}^2\right) + \epi \frac{1}{2}\normE{\cdot}^2 \subseteq \epi f \subseteq \epi \sigma \label{eq: focusGivesQuadUpperBoundAtZero}
    \end{align}
    where the first inclusion is by the above quadratic bound while the second follows by assumption. Our desired inclusion holds because of \eqref{eq: focusGivesQuadUpperBoundAtZero} and the next calculation, which shows that $\cCenter{\sigma} = - \nabla f(0)$ and $r_\sigma = f(0)- \frac{1}{2}\normE{\nabla f(0)}^2$. We have
    \begin{align*}
        r_\sigma + \frac{1}{2}\normE{\cCenter{\sigma}}^2 \leq f(0) - \frac{1}{2}\normE{\nabla f(0)}^2 + \frac{1}{2}\normE{\nabla f(0)}^2 \leq r_\sigma + \frac{1}{2}\normE{\cCenter{\sigma}}^2 
    \end{align*}
    where the first inequality holds because $\left(- \nabla f(0), f(0)- \frac{1}{2}\normE{\nabla f(0)}^2\right) \in \core{\sigma}$ by \eqref{eq: focusGivesQuadUpperBoundAtZero}, while the second follows because $f(0) = |f(0) - \sigma(0)| \leq \dist{f, \sigma} \leq \cGap{\sigma}$. Equation~\eqref{focusFormula} then ensures $\cCenter{\sigma} = -\nabla f(0)$ and $r_\sigma = f(0)- \frac{1}{2}\normE{\nabla f(0)}^2$ as claimed. 
\end{proof}

Now, by Lemmas~\ref{lem: bigCapSigmaInnerSmoothing}~and~\ref{lem: lowCapSigmaInnerSmoothing}, the necessity of the condition $\sIn{F}{\sigma} \leq f \leq \sIn{f}{\sigma}$ is proved. For sufficiency, we note that if $f$ is a convex inner $1$-smoothing satisfying $\sIn{F}{\sigma} \leq f \leq \sIn{f}{\sigma}$, then
\begin{align*}
    \cGap{\sigma} \leq \dist{\sIn{F}{\sigma}, \sigma} \leq \dist{f, \sigma} \leq \dist{\sIn{f}{\sigma}, \sigma} \leq \cGap{\sigma}
\end{align*}
where the first inequality follows by Lemma~\ref{lem: bigCapSigmaInnerSmoothing}, and the second and third inequalities follow by assumption, while the last inequality by Lemma~\ref{lem: lowCapSigmaInnerSmoothing}. Thus $\dist{f, \sigma} = \dist{\sIn{F}{\sigma}, \sigma}$, hence $f$ is optimal as Lemma~\ref{lem: bigCapSigmaInnerSmoothing} ensures $\sIn{F}{\sigma}$ is optimal. This concludes the proof of the second statement of the theorem.

The first part of the theorem also follows directly from Lemmas~\ref{lem: bigCapSigmaInnerSmoothing}~and~\ref{lem: lowCapSigmaInnerSmoothing}. Note that $\sigma$ is inner $\lambda$-smoothable for all $\lambda \geq \cGap{\sigma}$ because $\scaleOp{\sIn{f}{\sigma}}{1/\beta}$ is an inner $(\cGap{\sigma}, \beta)$-smoothing of $\sigma$ for all $\beta > 0$, by Lemma~\ref{lem: lowCapSigmaInnerSmoothing} and Lemma~\ref{lem: scalingOperation}. On the other hand, if $\sigma$ is $\lambda$-smoothable, then $\lambda \geq \cGap{\sigma}$ because $\cGap{\sigma} \leq \dist{f, \sigma} \leq \lambda$ for any $(\lambda, 1)$-smoothing $f$, where the first inequality follows by Lemma~\ref{lem: bigCapSigmaInnerSmoothing}.

\subsubsection{Proof of Theorem~\ref{thm: optimalConeInnerSmoothing}} The proof of this theorem follows the same style as that of Theorem~\ref{thm: optimalSigmaInnerSmoothing}. As in that proof, we restrict to the case $\beta = 1$ without loss of generality and break the proof into two key lemmas. The first, Lemma~\ref{lem: bigCapConeInnerSmoothing}, proves the inclusion $S \subseteq \sIn{S}{K}$ for any convex inner $1$-smoothing $S$. The second, Lemma~\ref{lem: lowCapConeInnerSmoothing}, proves the inclusion $\sIn{s}{K} \subseteq S$ for any optimal smoothing $S$.
\begin{lemma}\label{lem: bigCapConeInnerSmoothing}
    $\sIn{S}{K}$ is an inner $1$-smoothing of $K$ with $\dist{\sIn{S}{K}, K} \geq \cGap{K}$. Further, $S \subseteq \sIn{S}{K}$ for any closed convex inner $1$-smoothing $S$.
\end{lemma}
\begin{proof}
    Clearly $\sIn{S}{K} = \core{K} + B(0, 1) \subseteq K$ by definition of $\core{K}$. Since $\core{K}$ is convex and closed, $\sIn{S}{K}$ is $1$-smooth by \cite[Proposition 3]{liu2024gauges}. We also have
    \[
    \dist{\sIn{S}{K}, K} \geq \normE{P_{\sIn{S}{K}}(0)} = \normE{P_{\core{K}}(0) + P_{B_1}(- P_{\core{K}}(0))} = \normE{\cCenter{K} - P_{B_1}(-\cCenter{K})} = \normE{\cCenter{K}} - 1
    \]
    where the inequality holds because $0 \in K$, the first equality holds by Equation~\eqref{eqn: ballPlusSetProjection}, the second equality holds by definition of $\cCenter{K}$, while the last equality holds because $P_{B_1}(-\cCenter{K}) = -\frac{\cCenter{K}}{\normE{\cCenter{K}}}$ as $K \neq \varspace$ implies $\normE{\cCenter{K}} \geq 1$. Since $\cGap{K} = \normE{\cCenter{K}} - 1$, the claimed $\dist{\sIn{S}{K}, K} \geq \cGap{K}$ holds.
    Now let $S\subseteq K$ be $1$-smooth. Then there exists a closed convex set $C$ such that $S = C + B(0, 1)$, by \cite[Proposition 3]{liu2024gauges}. Since $C + B(0,1) = S \subseteq K$, it follows that $C \subseteq \core{K}$ by definition of $\core{K}$. Therefore $S = C + B(0, 1) \subseteq \core{K} + B(0,1) = \sIn{S}{K}$, and the proof is complete. 
\end{proof}

\begin{lemma}\label{lem: lowCapConeInnerSmoothing}
    $\sIn{s}{K}$ is an inner $1$-smoothing of $K$ with $\dist{\sIn{s}{K}, K} \leq \cGap{K}$. Further, $\sIn{s}{K} \subseteq S$ for any closed, convex optimal inner $1$-smoothing $S$.
\end{lemma}
To prove Lemma~\ref{lem: lowCapConeInnerSmoothing}, we use the following useful intermediate lemma.
\begin{lemma}\label{lem: distShrinksWithBallAddition}
    $\onedist{C}{S + B(0, r)} = \left(\onedist{C}{S} - r\right)_+$ for any closed, convex sets $C$ and $S$.
\end{lemma}
\begin{proof}[Proof of Lemma~\ref{lem: distShrinksWithBallAddition}]
    This holds by direct calculation as 
\begin{align*}
    \onedist{C}{S + B(0, r)} 
    &= \sup_{y \in C}\normE{y - P_S(y) - P_{B_r}(y - P_S(y))} \\
    &= \sup_{y \in C}\left(\normE{y - P_S(y)} - r\right)_+ \\
    &= \left(\onedist{C}{S} - r\right)_+
\end{align*}
where the first equality holds by \eqref{eqn: ballPlusSetProjection}, the second uses the simple fact that $\normE{z - P_{B_r}(z)} = (\normE{z} - r)_+$ for all $z \in \varspace$, while the last equality holds by continuity of the function $\alpha \mapsto (\alpha)_+ = \max\set{0, \alpha}$.
\end{proof}
\begin{proof}[Proof of Lemma~\ref{lem: lowCapConeInnerSmoothing}]
    $\sIn{s}{K}$ is $1$-smooth by \cite[Proposition 3]{liu2024gauges}. Noting that $\cCenter{K} \in \core{K}$ by definition of $\cCenter{K}$, we have
    \begin{align*}
        \sIn{s}{K} = x_K + B(0, 1) + K \subseteq K + K = K
    \end{align*}
    where the first equality follows by definition of $\sIn{s}{K}$ and the inclusion follows because $\cCenter{K} \in \core{K} \subseteq K$. We also have
    \begin{align*}
        \dist{\sIn{s}{K}, K} = \onedist{K}{\cCenter{K} + K + B(0, 1)} = \left(\onedist{K}{\cCenter{K} + K} - 1\right)_+ \leq \left(\normE{\cCenter{K}} - 1\right)_+ = \cGap{K}
    \end{align*}
    where the first equality holds because $\sIn{s}{K} = \cCenter{K} + K + B(0, 1) \subseteq K$, the second equality holds by Lemma~\ref{lem: distShrinksWithBallAddition}, the inequality holds because $\normE{y - P_{(\cCenter{K} + K)}(y)} \leq \normE{y - (\cCenter{K} + y)} = \normE{\cCenter{K}}$ for all $y \in K$, while the last equality holds by definition of $\cGap{K}$ and the fact that $\normE{\cCenter{K}} \geq 1$ as $\cCenter{K} + B(0,1) \subseteq K \neq \varspace$.

    Now suppose $S$ is a closed convex optimal inner $1$-smoothing. We will show that $B(\cCenter{K}, 1) \subseteq S$. This will imply $\sIn{s}{K} \subseteq S$ since
    \[\sIn{s}{K} = \cCenter{K} + K + B(0,1) = B(\cCenter{K}, 1) + S^{\infty} \subseteq S  \]
    where the second equality holds by Lemma~\ref{lem: distToConeFixesHorizon} as optimality of $S$ ensures $\dist{S, K} \leq \dist{\sIn{s}{K}, K} \leq \cGap{K} <\infty$. Let $\zeta \in N_S(P_S(0))$ have $\normE{\zeta} = 1$. Such $\zeta$ exists because $P_S(0) \in \bdry S$ since  $P_S(0) \in \interior S$ would imply $0 \in \interior S \subseteq \interior K$, which contradicts the assumption that $K \neq \varspace$. We have $B(P_{S}(0) - \zeta, 1) \subseteq S$ by $1$-smoothness of $S$. Therefore $B(\cCenter{K}, 1) \subseteq S$ will be proved once we show that $\cCenter{K} = P_{S}(0) - \zeta$. Observe that $P_S(0) - \zeta \in \core{K}$ because
    \begin{equation*}
        P_S(0) - \zeta + B(0, 1) = B(P_S(0) - \zeta, 1) \subseteq S \subseteq K
    \end{equation*}
    where the first inclusion follows by $1$-smoothness of $S$. Therefore
    \begin{equation}\label{coneDistToInnerSmoothing}
        \normE{\cCenter{K}} \leq \normE{P_{S}(0) - \zeta} = \normE{P_S(0)} + 1 \leq \dist{S, K} + 1 \leq \cGap{K} + 1 = \normE{\cCenter{K}}
    \end{equation}
    where the first inequality holds because $\cCenter{K} = P_{\core{K}}(0)$ and $P_S(0) - \zeta \in \core{K}$, while the first equality follows because $-P_S(0) = 0 - P_S(0) = \normE{P_S(0)} \zeta$, and the last inequality follows because $\dist{S, K} \leq \dist{\sIn{s}{K}, K}$ by optimality assumption. Now, \eqref{coneDistToInnerSmoothing} implies $\normE{P_S(0) - \zeta - 0} = \normE{\cCenter{K}} = \normE{P_{\core{K}}(0) - 0}$. Therefore $P_S(0) - \zeta = x_K$ by uniqueness of orthogonal projections. This proves that $B(\cCenter{K}, 1) \subseteq S$, thereby concluding the proof.
\end{proof}

From Lemmas~\ref{lem: bigCapConeInnerSmoothing}~and~\ref{lem: lowCapConeInnerSmoothing}, we immediately get the necessity of the condition $\sIn{s}{K} \subseteq S \subseteq \sIn{S}{K}$ on optimal smoothings. For sufficiency, note that if $S$ is a closed, convex inner $1$-smoothing that satisfies $\sIn{s}{K} \subseteq S \subseteq \sIn{S}{K}$, then 
\begin{align*}
    \cGap{K} \leq \dist{\sIn{S}{K}, K} \leq \dist{S, K} \leq \dist{\sIn{s}{K}, K} \leq \cGap{K}
\end{align*}
where the first inequality follows by Lemma~\ref{lem: bigCapConeInnerSmoothing}, the second and third inequalities hold because $\dist{\cdot, K}$ is monotone decreasing on subsets of $K$ due to $\dist{\cdot, K} = \onedist{K}{\cdot}$, while the last inequality holds by Lemma~\ref{lem: lowCapConeInnerSmoothing}. Thus $\dist{S, K} = \dist{\sIn{S}{K}, K}$, hence $S$ is optimal as Lemma~\ref{lem: bigCapConeInnerSmoothing} ensures $\sIn{S}{K}$ is optimal. This proves the second statement of the theorem.

The first statement of Theorem~\ref{thm: optimalConeInnerSmoothing} also follows immediately from Lemmas~\ref{lem: bigCapConeInnerSmoothing}~and~\ref{lem: lowCapConeInnerSmoothing}. Note $K$ is inner $\lambda$-smoothable for all $\lambda \geq \cGap{K}$ because the set $\scaleOp{\sIn{s}{K}}{1/\beta}$ is an inner $(\cGap{K}, \beta)$-smoothing, by Lemmas~\ref{lem: lowCapConeInnerSmoothing} and~\ref{lem: scalingOperation}. On the other hand, if $K$ is inner $\lambda$-smoothable, then $\lambda \geq \cGap{K}$ as any inner $(\lambda, 1)$-smoothing $S$ of $K$ has 
$\cGap{K} \leq \dist{S, K} \leq \lambda$ where the first inequality follows from Lemma~\ref{lem: bigCapConeInnerSmoothing}.

\subsubsection{Proof of Theorems~\ref{thm: optimalSigmaSmoothing} and~\ref{thm: optimalSigmaOuterSmoothing}}
Without loss of generality, we assume $\beta = 1$. Two key claims are needed for these proofs: for any convex, $1$-smooth function $f$ and any $r \geq 0$,
\begin{enumerate}[label=(\Alph*)]
    \item If $f \geq \sigma$, then $f - r$ is $1$-smooth and $\dist{f - r, \sigma} \leq \max\set{r, (\dist{f, \sigma} - r)_+}$.

    \item If $\dist{f, \sigma} \leq r$, then $f = f^{\mathrm{in}} - r$ where $f^{\mathrm{in}}$ is an inner $1$-smoothing  with 
    \[
    \dist{f^{\mathrm{in}}, \sigma} \leq \dist{(f - \sigma)_+, 0} + r .
    \]
\end{enumerate}
Before proving these statements, we show that they immediately yield the theorems. Since $f^{\mathrm{in}} \geq \sigma$ and $\dist{f^{\mathrm{in}}, \sigma} \leq \cGap{\sigma}$ for any optimal inner $1$-smoothing by Lemma~\ref{lem: lowCapSigmaInnerSmoothing}, (A) ensures that $\sigma$ is $\max\set{r, (\cGap{\sigma} - r)_+}$-smoothable with the smoothing $\sIn{f}{\sigma} - r$. Further, for $r \geq \cGap{\sigma}$, this smoothing is an outer smoothing. Now, by (B), for any $1$-smoothing $f$, we have $f = f^{\mathrm{in}} - \dist{f, \sigma}$ and 
\begin{align*}
    \cGap{\sigma} \leq \dist{f^{\mathrm{in}}, \sigma} \leq \dist{(f - \sigma)_+, 0} + \dist{f, \sigma} \leq 2 \cdot \dist{f, \sigma}
\end{align*}
where the first inequality follows by Lemma~\ref{lem: bigCapSigmaInnerSmoothing}, the second holds by (B), while the third holds because $\dist{f, \sigma} = \max\set{\dist{(f - \sigma)_+, 0}, \dist{(\sigma - f)_+, 0}}$. Similarly, noting that $(f - \sigma)_+ = 0$ if $f$ is an outer smoothing, (B) ensures $\dist{f, \sigma} \geq \cGap{\sigma}$ for any outer $1$-smoothing $f$. Since $\max\set{r, (\cGap{\sigma} - r)_+} = r$ for $r \in \set{\frac{\cGap{\sigma}}{2}, \cGap{\sigma}}$, we have therefore shown that with $r = \frac{\cGap{\sigma}}{2}$ and $r = \cGap{\sigma}$,  the transformation $f^{\mathrm{in}} \mapsto f^{\mathrm{in}} - r$ is a surjective map from optimal inner $1$-smoothings to optimal general and outer $1$-smoothings, respectively. Combining this observation with Theorem~\ref{thm: optimalSigmaInnerSmoothing} and Equations~\eqref{eq: reformulationOfOptFunctions} proves the if and only if characterization of the optimal general and outer smoothings.

We complete the proof by showing that (A) and (B) are true. For (A), clearly $f - r$ is as smooth as $f$. Further, if $f \geq \sigma$ then
\begin{align*}
    \dist{f - r, \sigma} 
    = \sup_{x \in \varspace} \max\set{r + \sigma(x) - f(x), f(x) - \sigma(x) - r} 
    & \leq \max\set{r, (\dist{f, \sigma} - r)_+}
\end{align*}
where the inequality uses that $f(x) \geq \sigma(x)$ by assumption and $\dist{f, \sigma} \geq f(x) - \sigma(x)$ for all $x \in \varspace$. For (B), if $f$ is $1$-smooth and $\dist{f, \sigma} \leq r$, then clearly $f^{\mathrm{in}} = f + r$ is $1$-smooth and $f^{\mathrm{in}} \geq \sigma$. Further
\begin{align*}
    \dist{f^{\mathrm{in}}, \sigma} 
    = \sup_{x \in \varspace} \left(f(x) + r - \sigma(x)\right)_+
    \leq \sup_{x \in \varspace} \left(f(x) - \sigma(x)\right)_+ + r 
    &= \dist{(f - \sigma)_+, 0} + r
\end{align*}
where the first equality holds because $f^{\mathrm{in}} = f + r \geq \sigma$ while inequality uses that $r \geq 0$.

\subsubsection{Proof of Theorems~\ref{thm: optimalConeSmoothing} and~\ref{thm: optimalConeOuterSmoothing}}
The proof mirrors the proof of Theorems~\ref{thm: optimalSigmaSmoothing}~and~\ref{thm: optimalSigmaOuterSmoothing}. As in that proof, we assume $\beta = 1$ with no loss in generality. We then show that the set transformation $S \mapsto\scaleOp{S + B(0, r)}{1/(1+r)}$ (which is analogous to $f \mapsto f - r$) is a surjective map from optimal inner $1$-smoothings and to general and outer optimal $1$-smoothings when $r = \frac{\cGap{K}}{2}$ and $r = \cGap{K}$, respectively. The key observations are again that for any closed, convex, $1$-smooth set $S$ and any $r \geq 0$,
\begin{enumerate}[label=(\Alph*)]
    \item If $S \subseteq K$, then $\scaleOp{S + B\left(0, r\right)}{1/(1 + r)}$ is $1$-smooth and 
    \[
    \dist{\scaleOp{S + B\left(0, r\right)}{1/(1 + r)}, K} \leq \frac{\max\set{r, \left(\dist{S, K} - r\right)_+}}{1 + r}.\]

    \item If $\dist{S, K} \leq \frac{r}{1 + r}$, then $S = \scaleOp{S^{\mathrm{in}} + B(0, r)}{1/(1+r)}$, where $S^{\mathrm{in}}$ is an inner $1$-smoothing of $K$ with
    \[
    \dist{S^{\mathrm{in}}, K} \leq (1+r)\onedist{K}{S} + r .
    \]
\end{enumerate}
Proving that these statements are true is sufficient, for entirely symmetric arguments as in the previous proof of Theorems~\ref{thm: optimalSigmaSmoothing}~and~\ref{thm: optimalSigmaOuterSmoothing}: smoothability claims follows from (A) by transforming an optimal inner $1$-smoothing using the appropriate choice of $r \in \set{\frac{\cGap{K}}{2}, \cGap{K}}$, while the if and only if characterizations follow because of Theorem~\ref{thm: optimalSigmaInnerSmoothing}, Equations~\eqref{eq: reformulationOfOptSets}, and the fact that (A) and (B) together guarantee that $S \mapsto \scaleOp{S + B(0, r)}{1/(1+r)}$ is a surjective map from optimal inner $1$-smoothings to optimal general (with $r = \frac{\cGap{K}}{2})$ and outer (with $r = \cGap{K}$) $1$-smoothings. 

For (A), note that $S + B\left(0, r\right)$ is $(1 + r)^{-1}$-smooth by \cite[Proposition 3]{liu2024gauges}. Therefore, by Lemma~\ref{lem: scalingOperation}, $\scaleOp{S + B(0, r)}{1/(1+r)}$ is $1$-smooth. Further
\begin{align*}
    \dist{S + B\left(0, r\right), K} 
    &= \max\set{\onedist{S+B\left(0, r\right)}{K}, \onedist{K}{S+B\left(0, r\right)}} \\ 
    &= \max\set{\onedist{S+B\left(0, r\right)}{K}, \left(\dist{S, K} - r\right)_+} \\ 
    & \leq \max\set{r, \left(\dist{S, K} - r\right)_+}
\end{align*}
where the second equality holds by Lemma~\ref{lem: distShrinksWithBallAddition} and the fact that $S \subseteq K$ ensures $\onedist{K}{S} = \dist{S, K}$; the inequality follows because $\onedist{S + B_r}{K} \leq \onedist{K + B_r}{K} = r$ due to $S \subseteq K$. Lemma~\ref{lem: scalingOperation} then ensures $\dist{\scaleOp{S + B(0, r)}{1/(1 + r)}, K} \leq \frac{\max\set{r, \left(\dist{S, K} - r\right)_+}}{1 + r}$ .

We now prove (B). Suppose $\dist{S, K} \leq \frac{r}{1 + r}$. Then Lemma~\ref{lem: scalingOperation} ensures $(1+r)S$ is $(1 + r)^{-1}$-smooth and $\dist{(1+r)S, K} \leq r$. Therefore, there exists a closed, convex $C$ such that $(1+r)S = C + B(0, 1 + r)$, by \cite[Proposition 3]{liu2024gauges}. Let $S^{\mathrm{in}} = C + B(0, 1)$. Clearly, $(1+r)S = S^{\mathrm{in}} + B(0, r)$. Further, $S^{\mathrm{in}}$ is $1$-smooth by \cite[Proposition 3]{liu2024gauges}. Note that $S^{\mathrm{in}} \subseteq K$ holds by the contrapositive since every $x \notin K$ has 
\begin{align*}
    \onedist{x + B\left(0,r\right)}{K} 
    &\geq \normE{x + r\frac{x - P_K(x)}{\normE{x - P_K(x)}} - P_K\left(x + r\frac{x - P_K(x)}{\normE{x - P_K(x)}}\right)} \\
    & = r + \normE{x - P_K(x)} \\
    & > \onedist{S^{\mathrm{in}} + B(0, r)}{K}
\end{align*}
where the equality follows because $P_K\left(x + r\frac{x - P_K(x)}{\normE{x - P_K(x)}}\right) = P_K(x)$. The claimed distance bound holds because
\begin{align*}
    \dist{S^{\mathrm{in}}, K} = \onedist{K}{S^{\mathrm{in}}} - r + r \leq \left(\onedist{K}{S^{\mathrm{in}}} - r\right)_+ + r &= \onedist{K}{(1 + r)S} + r \\
    &= (1 + r)\onedist{K}{S} + r
\end{align*}
where the first equality follows from $S^{\mathrm{in}} \subseteq K$, the inequality because $\alpha \leq (\alpha)_+$ for all $\alpha \in \RR$, the second equality follows by Lemma~\ref{lem: distShrinksWithBallAddition}, while the last equality holds because $K$ is invariant to positive rescalings.

%% file: applications.tex
\section{Applications: Nonsmooth Functions and Sets}\label{sec: conicSections}
Our theory can be used to handle nonsmoothness in more general settings. In this section, we demonstrate how this can be done for nonsmooth composite functions and nonsmooth sets with a nonempty interior.

\subsection{Beyond Sublinear Functions: Application to Amenable Functions} \label{subsec:amenable}
We consider the smoothing of convex functions $g$ of the form $g = \sigma \circ G$,
under the assumption that 
\begin{equation}\label{eq: amenableFunctions}
    \norm{G(x) - G(y)}_{\varspace^\prime} \leq M\normE{x - y} \quad \text{and} \quad \norm{J(x) - J(y)}_{\varspace \to \varspace^\prime} \leq L\normE{x - y} \quad \text{for all } x, y \in \varspace.
\end{equation}
Here $J(\cdot)$ denotes the Jacobian of $G: \varspace \to \varspace^\prime$ and $\norm{\cdot}_{\varspace \to \varspace^\prime}$ is the operator norm. Nonsmooth functions of this type have been studied widely in optimization \cite{Burke1985, Nesterov2013, Lewis2016} and analysis \cite{rockafellar2009, shapiro2003}. Our class of functions is a special subclass of amenable functions of Rockafellar and Wets \cite[Definition 10.23]{rockafellar2009}, where $\sigma$ is not assumed to be sublinear or finite everywhere and $G$ is only required to have a continuous Jacobian. Our class is also closely related to decomposable functions of Shapiro (see \cite[Definition 1.1]{shapiro2003}).

This class includes many nonsmooth functions that arise naturally in optimization; for example, minimizing finite maximums $\max\{G_1(x), \dots, G_n(x)\}$, maximum eigenvalue optimization $\lambda_{\max} (G(x))$, and nonlinear regression with an arbitrary norm $\norm{G(x)}$.

Using our theory, amenable functions satisfying \eqref{eq: amenableFunctions} can be smoothed by replacing the sublinear component $\sigma$ with any smoothing. Doing so optimally gives the following guarantee on the smoothability of such functions.
\begin{theorem}\label{thm: smoothnessOfComposition}
    Consider any function $g = \sigma \circ G$ where $\sigma$ is sublinear and $M_\sigma$-Lipschitz and $G:\varspace \to \varspace^\prime$ satisfies \eqref{eq: amenableFunctions}. Then $g$ is $(M^2w_\sigma/2, M_\sigma L)$-smoothable, given by the $\beta$-smoothing
    $$x\mapsto \frac{M^2}{\beta - M_\sigma L} \sGen{f}{\sigma}\left(\frac{\beta - M_\sigma L}{M^2}G(x)\right) \ . $$
\end{theorem}
\begin{proof}
    First, we observe that for any $\beta'$-smoothing $f_\sigma$ of $\sigma$, $f=f_\sigma \circ G$ is a $M_{\sigma}L + M^2\beta'$-smooth approximation to $g$ with $\dist{f_\sigma \circ G, g} \leq \dist{\sigma, f_\sigma}$. To see this, note that
    \begin{align*}
        \dist{f_\sigma \circ G, g} = \sup_{x \in \varspace}\abs{\sigma(G(x)) - f_\sigma(G(x))} \leq \dist{\sigma, f_\sigma} .
    \end{align*}
    Letting $J(\cdot)^T$ denote the adjoint operator of $J(\cdot)$, which is also $L$-Lipschitz continuous, we have
    \begin{align*}
        \normE{\nabla f(x) - \nabla f(y)} &= \normE{J(x)^T \nabla f_\sigma(G(x)) - J(y)^T\nabla f_\sigma(G(y))} \\
        & = \normE{[J(x)^T - J(y)^T]\nabla f_\sigma(G(x)) + J(y)^T[\nabla f_\sigma(G(x))- \nabla f_\sigma(G(y))]} \\
        & \leq \norm{\nabla f_\sigma(G(x))}_{\varspace^\prime}\norm{J(x)^T - J(y)^T}_{\varspace^\prime \to \varspace} + \norm{J(y)^T}_{\varspace^\prime \to \varspace}\norm{\nabla f_\sigma(G(x))- \nabla f_\sigma(G(y))}_{\varspace^\prime}\\
        & \leq \left(M_\sigma L + M^2\beta'\right)\normE{x - y}
    \end{align*}
    where the first equality follows by the chain rule, the first inequality uses the triangle inequality and the definition of operator norm while the last inequality follows because that $F, J(\cdot)^T$ and $\nabla f_\sigma$ are Lipschitz continuous. In particular, the $M$-Lipschitz continuity of $G$ ensures that $\norm{J(y)^T}_{\varspace^\prime \to \varspace} \leq M$. Moreover, the $M_\sigma$-Lipschitz continuity of $\sigma$ ensures that $\normE{\nabla f_\sigma(G(x))} \leq M_\sigma$ since $\dist{\sigma, f_\sigma} < \infty$ implies $\epi \sigma = (\epi f_\sigma)^\infty$ by Lemma~\ref{lem: distToConeFixesHorizon} and consequently  $\nabla f_\sigma(z) \in \partial \sigma(0)$ for all $z \in \varspace'$.

    From this bound, for any target smoothness $\beta > M_\sigma L$, consider selecting $f_\sigma$ as an optimal $\beta' = (\beta - M_\sigma L)/M^2$-smoothing of $\sigma$. This choice suffices to yield an $f=f_\sigma\circ G$ proving the claimed smoothability of $g=\sigma \circ G$.
\end{proof}
\begin{remark}
    By replacing $\sGen{f}{\sigma}$ with $\sIn{f}{\sigma}$, the same arguments can be used to show that $\sigma \circ G$ is inner $(M^2\cGap{\sigma}, M_{\sigma}L)$-smoothable. Similarly, $\sOut{f}{\sigma}$ can be used to show that $\sigma \circ G$ is outer $(M^2\cGap{\sigma}, M_{\sigma}L)$-smoothable.
\end{remark}

\subsubsection{Example: Finite Maximums and Improving LogSumExp Smoothing}
An important application where smoothing enables accelerated convergence theory is the minimization of a finite maximum
$$ \min_x \max \{G_1(x), \dots, G_n(x)\} \ . $$

Supposing each $G_i$ is convex and that $G(x) = (G_1(x), \dots, G_n(x))$ is $M$-Lipschitz, a subgradient method could be applied, giving a convergence guarantee of the form $M^2\normE{x_0 - x_\star}^2/\varepsilon^2$, where $x_0$ denotes the initial iterate and $x_\star$ is an optimal solution. Beck and Teboulle~\cite{Beck2012} considered replacing $\sigma(x) = \max\{y_1, \dots, y_n\}$ above by the smoothing $\eta f^{\mathrm{exp}}(G(x)/\eta)$ for small $\eta > 0$, where
$$
f^{\mathrm{exp}}(y) = \log\left(\sum_{i=1}^n \exp(y_i)\right) \ .
$$
Note that $f^{\mathrm{exp}}$ is $1$-smooth and $\dist{\sigma, f^{\mathrm{exp}}} = \log(n)$. Therefore, if the Jacobian of $G$ is $L$-Lipschitz, Lemma~\ref{lem: scalingOperation} and Theorem~\ref{thm: smoothnessOfComposition} ensure that $\eta f^{\mathrm{exp}}(G(x)/\eta)$ is $L + M^2/\eta$-smooth with $\max_{x \in \varspace}|\sigma(G(x)) - \eta f^{\mathrm{exp}}(G(x)/\eta)| \leq \eta \log(n)$. By choosing $\eta = \varepsilon / (2\log(n))$ and minimizing the smoothed objective to accuracy $\varepsilon/2$ using an accelerated gradient method (see \cite{nesterov1983method}), this yields an improved convergence guarantee of
$$
\sqrt{2L\varepsilon + 4\log(n)M^2}\cdot\frac{\norm{x_0 - x_\star}_2}{\varepsilon} \ .
$$

Note $f^{\mathrm{exp}}$ is not the optimal inner $1$-smoothing of $\sigma$. Rather the optimal inner smoothing is $\sIn{f}{\sigma} = \sGen{f}{\sigma}  + \frac{1}{4}(1 - \frac{1}{n})$ where $\sGen{f}{\sigma}$ is given by~\eqref{eq: maxFuncOptSmoothing}. Using this smoothing and applying an accelerated gradient method to minimize $\varepsilon \sIn{f}{\sigma}(G(x)/\varepsilon)$ to accuracy $\varepsilon /2$ gives an overall convergence guarantee of
$$
\sqrt{2L\varepsilon + 2M^2}\cdot\frac{\norm{x_0 - x_\star}_2}{\varepsilon} \ .
$$
This is an improvement by a factor of $1/\sqrt{\log(n)}$ in iteration complexity of prior works using the $f^{\mathrm{exp}}$ smoothing. In terms of per-iteration computational complexity, these methods are comparable: At each step, both methods must compute a Jacobian of $G$. Although computing the optimal smoothing~\eqref{eq: maxFuncOptSmoothing} and its gradient is somewhat more complex than those of the logSumExp, it can still be done in (expected) linear time in $n$~\cite{Condat2016}. 

We note that the above improvement is not universal. This is because $f^{\mathrm{exp}}$ is also $1$-smooth with respect to infinity norm, i.e $\norm{\nabla f^{\mathrm{exp}}(x) - f^{\mathrm{exp}}(y)}_1 \leq \norm{x - y}_{\infty}$. Therefore, in calculating the smoothness constant of $f^{\mathrm{exp}} \circ G$, one could instead use the $\tilde{M}$-Lipschitz and $\tilde{L}$-smooth constants under the infinity norm on $\RR^n$. Recall, our smoothings are only optimal with respect to the two-norm. Since $\norm{\cdot}_\infty \leq \norm{\cdot}_2,$ it follows that $\tilde{M} \leq M$ and $\tilde{L} \leq L$. Using this potentially more refined smoothness constant for $f^{\mathrm{exp}}$, our theory provides an improved rate whenever $M/\tilde{M} \leq \sqrt{2\log(n)}$.

To attain a universal improvement on the $f^{\mathrm{exp}}$ smoothing, new theory is needed for optimal smoothings under general norms and, in particular, the infinity norm. We leave open the question of designing optimal smoothings under norms that do not have an inner product.

\subsection{Beyond Convex Cones: Application to General Convex Sets}\label{subsec: coneApplication}
It is well-known that for any convex set $C \subseteq \varspace$, there exists a cone $K \subseteq \varspace^\prime \defeq \varspace \times \RR$ and $x_0 \in \varspace$ such that $C = \set{x \in \varspace \mid (x - x_0, 1) \in K}$. Therefore, by smoothing $K$ optimally, our theory can be used to approximate generic convex sets $C$. The following Theorem~\ref{thm: affineInSmoothableConeDescriptionOfConvexSets} shows that $K$ can be chosen to be $\left(1 - \frac{R}{\sqrt{1 + R^2}}\right)$-smoothable whenever $B(x_0, R) \subseteq C$. Doing so, our optimal smoothings for $K$ strictly improve upon the trivial $1$-smoothing of $C+B(0,1/\beta)$, at the cost of growing the problem dimension by one. These smoothings may be useful in optimization, as recent works \cite{liu2024gauges,luner2024performanceestimationsmoothstrongly,samakhoana2024scalableprojectionfreeoptimizationmethods} have proposed algorithms that benefit from the smoothness of constraints.
\begin{theorem}\label{thm: affineInSmoothableConeDescriptionOfConvexSets}
    Let $C \subseteq \varspace$ be a closed, convex set with $B(x_0, R) \subseteq C$ for some $R > 0$. Then 
    \[
    C = \set{x\in \varspace \mid (x - x_0, 1) \in K}
    \]
    where $K$ is the closure of the cone $\set{(x,r) \in \varspace \times \RR \mid r > 0, x/r \in C - x_0}$. Further, $K$ is closed, convex, and outer $\left(1 - \frac{R}{\sqrt{1 + R^2}}\right)$-smoothable.
\end{theorem}
\begin{proof}
    We only prove smoothability since the other statements are obvious. By Theorem~\ref{thm: optimalConeSmoothing}, it suffices to show that $\frac{\cGap{K}}{1 +\cGap{K}} \leq 1 - \frac{R}{\sqrt{1 + R^2}}$. Let $\varspace^\prime = \varspace \times \RR$. Recalling that $\cGap{K} = \norm{P_{\core{K}}(0)}_{\varspace^\prime} - 1$, and noting that $\frac{\cGap{K}}{1 +\cGap{K}} \leq 1 - \frac{R}{\sqrt{1 + R^2}} \iff \cGap{K} \leq \frac{\sqrt{1 + R^2} - R}{R}$, it suffices to show that there exists $(\bar x, \bar r) \in \core{K}$ such that $\norm{(\bar x, \bar r)}_{\varspace^\prime} - 1\leq \frac{\sqrt{1 + R^2} - R}{R}$. Taking $(\bar x, \bar r) = (0, \sqrt{1 + R^2}/R)$, we have
    \begin{align*}
        \norm{(\bar x, \bar r)}_{\varspace^\prime} - 1 = \frac{\sqrt{1 + R^2}}{R} - 1 = \frac{\sqrt{1 + R^2} - R}{R}.
    \end{align*}
    Further, we have
    \begin{align*}
        (\bar x, \bar r) \in \core{K} 
        &\iff (\bar x, \bar r) + B(0, 1) \subseteq K \\
        &\iff \left[\normE{x}^2 + \left(r - \frac{\sqrt{1 + R^2}}{R}\right)^2 \leq 1  \implies (x, r) \in K\right] \\
        & \impliedby \left[\normE{x}^2 + \left(r - \frac{\sqrt{1 + R^2}}{R}\right)^2 \leq 1 \implies r > 0  \; \text{and} \; \normE{\frac{x}{r}} \leq R\right] \\
        & \impliedby \left[\normE{x}^2 + \left(r - \sqrt{1 + R^{-2}}\right)^2\leq 1 \implies  \normE{x}^2 \leq r^2R^2\right] \\
        & \iff \left[\normE{x}^2 \leq r^2R^2 - \left(r^2R^2 + \left(r - \sqrt{1 + R^{-2}}\right)^2 - 1 \right) \implies \normE{x}^2 \leq r^2R^2\right] \\
        & \impliedby r^2R^2 + \left(r - \sqrt{1 + R^{-2}}\right)^2 - 1 \geq 0 \\
        & \iff (1 + R^2)r^2 - 2r\frac{\sqrt{1 + R^2}}{R} + \frac{1}{R^2} \geq 0 \\
        & \iff \left(r\sqrt{1+R^2} - \frac{1}{R}\right)^2 \geq 0
    \end{align*}
    where the first line follows by definition of $\core{K}$, the third line follows by definition of $K$ and the assumption that $B(0, R) \subseteq C - x_0$, and the fourth because $\normE{x}^2 + \left(r - \sqrt{1 + R^{-2}}\right)^2 \leq 1$ clearly implies $r \geq \sqrt{1 + R^{-2}}-  1 > 0$. Therefore, $(\bar x, \bar r) \in \core{K}$ and the proof is complete.
    
\end{proof}

%% file: appendix.tex


\section{Optimal Smoothing of the Maximum Eigenvalue Function}\label{app:max-eigen}
    Consider $\varspace = \mathbb{S}^{d\times d}$, the space of symmetric $  d\times d$ matrices with the trace inner product and Frobenius norm. Let $\sigma(A) \defeq \lambda_{\max}(A) \defeq \max_{x \in \RR^d, \norm{x}_2 =1}x^TAx$. Note that $\partial \sigma(0) = \mathrm{conv}\{xx^T \mid x \in \RR^d, \norm{x}_2 =1\}$. Therefore, the functional core and center are
    \begin{align*}
        & \price{\sigma}(A) = \max_{\zeta \in \partial \sigma(0)}\ip{\zeta, A} + \frac{1}{2}\normE{\zeta}^2 = \max_{\zeta \in \{xx^T \mid x \in \RR^d, \norm{x}_2 =1\}}\ip{\zeta, A} + \frac{1}{2} = \sigma(A) + \frac{1}{2}, \\
        & \cCenter{\sigma} = \argmin_{A \in \mathbb{S}^{d\times d}}\frac{1}{2} + \sigma(A) + \frac{1}{2}\normE{A}^2 =  - I_d / d, \\
        & r_\sigma = - \frac{1}{d} + \frac{1}{2}
    \end{align*}
    where $I_d$ is the identity matrix for $\RR^d$. This implies $\price{\sigma} = r_\sigma + \sigma(\cdot - \cCenter{\sigma})$ and so $\sigma$ has a unique optimal $1$-smoothing. In particular, it is given by\footnote{This can be verified by directly computing $\infConv{\price{\sigma}}{\frac{1}{2}\normE{\cdot}^2} - \cGap{\sigma}/2 = \infConv{\left(\frac{1}{2} + \sigma\right)}{\frac{1}{2}\normE{\cdot}^2} -(1 - 1/d)/4 $ using Equation~\eqref{proxOfSigma} and noting that $P_{\partial\sigma(0)}(A) = \sum_{i=1}^d(\lambda_i(A)- \alpha)_+ v_iv_i^T$ if $A = \sum_{i=1}^d\lambda_i(A) v_iv_i^T$ is the spectral decomposition of $A$.}
    \begin{equation*}
        \sGen{f}{\lambda_{\max}}(A) = \infConv{\lambda_{\max}}{\frac{1}{2}\normE{\cdot}^2}(A) + \frac{1}{4}\left(1 + \frac{1}{d}\right) = \alpha + \frac{1}{2}\sum_{j = 1}^J(\lambda_j(A) - \alpha)^2 + \frac{1}{4}\left(1 + \frac{1}{d}\right) 
    \end{equation*}
    where $\lambda_1(A), \dots, \lambda_d(A)$ are the eigenvalues of $A$ sorted in descending order, $\alpha =  \frac{1}{J}(\sum_{i = 1}^J\lambda_i(A) - 1)$, and
    $J = \max\set{j \in \{1, \dots, d\} \mid \lambda_j(A) - \frac{1}{j}(\sum_{i = 1}^{j}\lambda_i(A) - 1) > 0}$.
    
    The error of approximation is $\max_{A \in \mathbb{S}^{d\times d}}\abs{\sGen{f}{\lambda_{\max}}(A) - \sigma(A)} = \frac{1}{2}\cGap{\sigma} = \frac{1}{4}(1 -\frac{1}{d})$. More generally, $\frac{1}{\beta}\sGen{f}{\lambda_{\max}}(\beta A)$ gives the optimal $\beta$-smoothing for any $\beta > 0$, with approximation error of $\frac{1}{4}(1 -\frac{1}{d})/\beta$.